\documentclass[12pt,reqno]{amsart}
\usepackage{mathrsfs}
\usepackage{amsgen}
\usepackage{amscd}
\usepackage{amsmath}
\usepackage{latexsym}
\usepackage{amsfonts}
\usepackage{amssymb}
\usepackage{amsthm}
\usepackage{graphicx}
\usepackage{color}
\usepackage{amsthm,amssymb}
\usepackage{graphicx}
\usepackage{enumerate}
\usepackage{amssymb,amsmath,graphicx,amsfonts,euscript}
\usepackage{color}
\usepackage{mathrsfs}

\usepackage[margin=3cm, a4paper]{geometry}

\def\H{{\cal H}}

\def\R{\mathbb{R}}
\def\Z{\mathbb{Z}}

\def\C{\mathbb{C}}

\def\H2{H^2(\R^N)}
\def\L2{L^2(\R^N)}
\def\to{\rightarrow}
\def\cd{\!\cdot\!}

\def\H{{\cal H}}

\def\cd{\!\cdot\!}

\def\H1{H^1(\R)}

 \newcommand{\Del}[1]{}

\numberwithin{equation}{section}

\newtheorem{thm}{Theorem}[section]
\newtheorem{cor}[thm]{Corollary}
\newtheorem{lem}[thm]{Lemma}

\newtheorem{prop}[thm]{Proposition}

\theoremstyle{remark}
\newtheorem{remark}[thm]{Remark}
\newtheorem*{exam*}{Examples}

\begin{document}

\setcounter{page}{1}

\title[Large global solutions for NLS, I]{Large global solutions for nonlinear Schr\"odinger equations I, mass-subcritical cases}

\author{Marius Beceanu}
\address{Department of Mathematics and Statistics\\
University at Albany SUNY\\
Earth Science 110\\
Albany, NY, 12222, USA\\}
\email{mbeceanu@albany.edu}
\thanks{}

\author{Qingquan Deng}
\address{Department of Mathematics\\
Hubei Key Laboratory of Mathematical Science\\
Central China Normal University\\
Wuhan 430079, China.\\}
\email{dengq@mail.ccnu.edu.cn}
\thanks{}

\author{Avy Soffer}
\address{Department of Mathematics\\
Rutgers University\\
110 Frelinghuysen Rd.\\
Piscataway, NJ, 08854, USA\\}
\email{soffer@math.rutgers.edu}
\thanks{}

\author{Yifei Wu}
\address{Center for Applied Mathematics\\
Tianjin University\\
Tianjin 300072, China}
\email{yerfmath@gmail.com}
\thanks{}




\maketitle

\begin{abstract}\noindent
In this paper, we consider the nonlinear Schr\"odinger equation,
$$
    i\partial_{t}u+\Delta u= \mu|u|^p u, \quad (t,x)\in \R^{d+1},
$$
with $\mu=\pm1, p>0$.

In this work, we consider the mass-subcritical cases, that is, $p\in (0,\frac4d)$. We prove that under some restrictions on $d,p$, any radial initial data in the critical space $\dot H^{s_c}(\R^d)$ with compact support, implies  global well-posedness.
\end{abstract}

\tableofcontents

\section{Introduction}
We study the Cauchy problem for the following nonlinear
Schr\"{o}dinger equation (NLS) on $\R\times\R^d$:
 \begin{equation}\label{eqs:NLS-cubic}
   \left\{ \aligned
    &i\partial_{t}u+\Delta u=\mu |u|^p u,
    \\
    &u(0,x)  =u_0(x),
   \endaligned
  \right.
 \end{equation}
with $\mu=\pm1, p>0$.
Here $u(t,x):\R\times\R^d\rightarrow \C$ is a complex-valued function. The case $\mu=1$ is referred to the defocusing case, and the case $\mu=-1$ is the focusing case. The class of solutions to equation (\ref{eqs:NLS-cubic}) is invariant under the scaling
\begin{equation}\label{eqs:scaling-p}
u(t,x)\to u_\lambda(t,x) = \lambda^{\frac2p} u(\lambda^2 t, \lambda x) \ \ {\rm for}\ \ \lambda>0,
\end{equation}
which maps the initial data as
\begin{eqnarray}
u(0)\to u_{\lambda}(0):=\lambda^{\frac2p} u_{0}(\lambda x) \ \ {\rm for}\ \ \lambda>0.\nonumber
\end{eqnarray}
Denote
$$
s_c=\frac d2-\frac2p.
$$
Then the scaling  leaves  $\dot{H}^{s_{c}}$ norm invariant, that is,
\begin{eqnarray*}
\|u\|_{\dot H^{s_{c}}}=\|u_{\lambda}\|_{\dot H^{s_{c}}},
\end{eqnarray*}
which is called \emph{critical regularity} $s_{c}$. It is also considered as the lowest regularity for which the problem  (\ref{eqs:NLS-cubic}) is well-posed for general $H^{s}(\R^d)$-data. Indeed,  it was proved by Christ, Colliander, Tao \cite{ChCoTa-Ill} that  there exist some initial datum belonging to $H^s(\R^d), s<s_c$ such that the problem   (\ref{eqs:NLS-cubic}) is ill-posed.

The $H^1$-solution of equation \eqref{eqs:NLS-cubic} also enjoys  mass, momentum and energy
conservation laws, which read
\begin{equation}\label{eqs:energy-mass}
   \aligned
M(u(t))&:=\int |u(t,x)|^2\,dx=M(u_0),\\
P(u(t))&:=\textrm{Im}\int \overline{u(t,x)}\nabla u(t,x)\,dx=P(u_0),\\
E(u(t)) &:= \int |\nabla u(t,x)|^2\,dx + \frac{2\mu}{p+2}\int
|u(t,x)|^{p+2} \,dx = E(u_0).
   \endaligned
\end{equation}

The well-posedness and scattering theory for Cauchy problem (\ref{eqs:NLS-cubic}) with initial data in $H^{s}(\R^d)$ were extensively studied, which we here  briefly review.
The local well-posedness theory follows from a standard fixed point argument, implying that  for all $u_{0}\in H^{s}(\R^d)$, there exists $T_{0}>0$ such that its corresponding solution $u\in C([0,T_{0}),\ H^{s}(\R^d))$. In fact, the above $T_{0}$ depends on $\|u_{0}\|_{H^{s}(\R^d)}$ when $s>s_c$ and also the profile of $u_{0}$  when $s=s_c$. Some of the results can be found in Cazenave and Weissler \cite{CW1}.

Such argument can be applied directly to prove the global well-posedness for solutions to equation (\ref{eqs:NLS-cubic})  with small initial data in $H^{s}(\R^d)$ with $s\geq s_c$.
In the mass-subcritical cases, that is, $p<\frac4d$, if we consider the solution in  $L^2(\R^d)$ space, the local theory above, together with the mass  conservation laws  (\ref{eqs:energy-mass}), yields the global well-posedness for any initial data $u_0\in L^2(\R^d)$. In the mass-supercritical, energy-subcritical cases, that is, $\frac4d<p<\frac4{d-2}$, if we consider the solution in energy space $H^1(\R^d)$, the local theory above together with  conservation laws  (\ref{eqs:energy-mass}) yields the global well-posedness for all initial data $u_0\in H^1(\R^d)$ in the defocusing case $\mu=1$, and for any initial data $u_0\in H^1(\R^d)$ with some restrictions in the focusing case.
Furthermore, the scattering under the same conditions were also obtained by Ginibre, Velo \cite{GiVe} in the defocusing case and \cite{DuHoRo} in the focusing case.
In the mass-critical and energy-critical cases, since the conservation laws do not imply directly the global existence of the solutions, the problem becomes much more complicated.
In the energy-critical case, the global well-posedenss and scattering in the defocusing case was first proved by Bourgain \cite{Bou2} in the radial data case and then by Colliander, Keel, Takaoka, Staffilani and Tao \cite{CKSTT08} in the non-radial data case in dimension three, the higher dimension cases were solved by Ryckman and Visan \cite{RV07AJM} and Visan \cite{Vi-1,Vi-2}; the global well-posedenss and scattering in the focusing case was proved by Kenig and Merle \cite{KeMe--NLS-2006} in the radial data case,  then by Killip, Visan \cite{KV-10-3} in the non-radial case when the dimensions are five and higher, and by Dodson \cite{Dodson-14} in four dimensions, see also \cite{{KV-10-4,Tao-1, TaViZh, Vi-1, Vi-2}} for some previous works and simplified proofs. In the mass-critical case, the global well-posedenss and scattering was first proved by Killip, Tao, Visan \cite{KiTaVi-2009} in the radial data case in dimension two, and Killip, Visan, Zhang \cite{KVZ-08} in dimensions higher than two, then in the non-radial data case, the problem was solved in a series of papers of Dodson \cite{Dodson-12, Dodson-16-1, Dodson-16-2, Dodson-15}.

More complicated situation appears if one considers the general nonlinear Schr\"odinger equations in the critical space $\dot H^{s_c}(\R^d)$.
Recently, conditional global and scattering results with the assumption of $u\in L^\infty_t(I,\dot H^{s_c}_x(\R^d))$ (here $I$ is the maximal lifespan) were considered by many authors, which was started from \cite{KeMe-cubic-NLS-2010, KeMe-wave-2011-1}, and then developed by
\cite{Bu, DoMiMuZh-17,DKM, DuRo, KeMe-wave-2011-2, KiMaMuVi-NoDEA-2017, KiMaMuVi-2018, KV, KV-10, KV-10-3, MJJ, MWZ, Mu, Mu-3, XieFa-13} and cited references.
 That is,  if the initial data $u_{0}\in \dot H^{s_c}(\R^d)$ and the solution has priori estimate
\begin{align}
\sup_{0<t<T_{out}(u_{0})}\|u\|_{\dot H^{s_c}_x(\R^d)}<+\infty, \label{uniformbound}
\end{align}
then $T_{out}(u_{0})=+\infty$ and the solution scatters in $\dot H^{s_c}(\R^d)$, here $[0,T_{out}(u_{0}))$ is the maximal interval in positive direction for existence of the solution. Consequently, these results give the blowup criterion which the lifetime  depends only on the critical norm $\|u\|_{L^\infty_t\dot H^{s_c}_x(I\times\R^d)}$.  However, it seems that no such large data global
results are known, if only the initial data $u_{0}\in \dot H^{s_c}(\R^d)$.
Furthermore,  many authors considered the large global solutions for rough data from a probabilistic point of view,
that is, one may construct a large sets of initial data of super-critical regularity which leads to global solutions,
see \cite{Bou-1994, Bou-1997, BuTz, BuTz-2, BuTz-3, CoOh-2012,Deng-2012, DeTz-2015, Dodson-17,KiMuVi-2017, LuMe-2014, NaOhReSt-2012, NaOhReSt-2015, OhOkPo-2017, OhPo-2016, Po-2017, PoRoTh-2014, Th-09}.

%


In the first part of our series of works, we consider the global solution for the mass-subcritical nonlinear Sch\"odinger equation in the critical space $\dot H^{s_c}(\R^d)$. Due to the mass conservation law, $L^2$-initial datum lead to the global solutions. It is known from Christ, Colliander and Tao \cite{ChCoTa-Ill} and Kenig, Ponce, Vega \cite{KePoVe-Duke-2001} that the problem is ill-posed in some sense for the non-radial datum in $\dot H^{s}(\R^d), s<0$. However, for the radial data, due to the better radial Strichartz estimates, one may establish the  local well-posedness result in negative regularity Sobolev spaces.
Indeed, it was proved by Guo and Wang \cite{GW} that there exists $p_0(d)<\frac4d$, such that for any $p\in (p_0(d),\frac4d)$, if the initial datum are radial and small in the critical space $\dot H^{s_c}(\R^d)$, then the nonlinear solutions of \eqref{eqs:NLS-cubic} are global and scatter.
Very recently, Killip, Masaki, Murphy and Visan \cite{KiMaMuVi-NoDEA-2017,KiMaMuVi-2018} proved a conditional result; that in the defocusing case, there exists $p_0(d)<\frac4d$, such that for any $p\in (p_0(d),\frac4d)$, if the radial solution $u\in L^\infty_t\dot H^{s_c}_x(I\times\R^d)$, then $I=\R$ and the solution scatters, by using concentration-compactness arguments.
This is the first global result for large data theory in the critical spaces for the mass-subcritical NLS.

In this paper, we prove unconditional global well-posedness. We prove that  for radial initial data with compact support in space, and is in the critical space, there exists solution global  in time.
\begin{thm}\label{thm:main01}
Let $d\ge 4$, and $\mu=\pm1$. Then there exists $p_0(d)\in (0,\frac4d)$, such that for any $p\in [p_0(d),\frac4d)$, the following is true. Suppose that  $u_0\in \dot H^{s_c}(\R^d)$ is a radial function satisfying
$$
\mbox{supp }u_0\subset \{x:|x|\le 1\}.
$$
Then the solution $u$ to the equation \eqref{eqs:NLS-cubic} with the initial data $u_0$
exists globally in time, and $u\in C(\R^+;\dot H^{s_c}(\R^d))\cap L^\infty(\R^+;\dot H^{s_c}(\R^d)+L^2(\R^d))$. Moreover, for any $t\in \R$,
$$
\big\| u(t)\big\|_{\dot H^{s_c}(\R^d)}
\lesssim 1+|t|.
$$
\end{thm}

\begin{remark}
We make several remarks regarding the above statements.

(1) Our conclusions are valid for both of the focusing and the defocusing cases.
Further, by scaling, one can extend the size of the radius 1 to an arbitrary large number.
Moreover, the compact support assumption on initial data are not necessary and can be replaced by
some weighted assumption.

(2) In the present paper, we are not going to give the sharp conditions on $p_0(d)$ and $d$.

In the mass-subcritical cases,   there is a new difficulty when we consider the global solution in the negative Sobolev space. It is worth noting that in this case, we can not use the mass, energy conservation laws, and Morawetz estimates. Moreover, the pseudo-conformal conservation law has no good sign.

Further, because all of the conservation laws are beyond the critical scaling regularity, we believe that analogous scattering result in $\dot H^{s_c}(\R^d)$ is very hard to pursue in the mass-subcritical case (it is similar to the energy-supercritical case in which all the conservation laws are below the critical scaling regularity), even if the initial data is smooth enough.
\end{remark}

\textbf{Sketch of the proof:}

First, in step 1, we show an improved (supercritical) Strichartz estimates for the initial data localized in space under the linear flow. More precisely, we prove that
for all $N\ge 1$, there exist $\alpha_0>1,\beta_0>0$, such that
\begin{align*}
\Big\|\langle t^{\alpha_0}|\nabla|\rangle^{\beta_0}|\nabla|^{s_c}\big(e^{it\Delta}\chi_{\le 10}(P_{\ge N}g)\big)\Big\|_{L^2_tL^\frac{2d}{d-2}_x(\R\times\R^d)}
\lesssim \|P_{\ge N}g\|_{\dot H^{s_c}(\R^d)}
\end{align*}
(a slight stronger estimate is needed, see Section \ref{Sec:thm01} below). From this estimate, we gain the regularity and time decay for $t\gtrsim 1$.

In step 2, given small constant $\delta_0>0$, we break the initial data into two parts, $u_0=v_0+w_0$, with
$$
v_0=\chi_{\le 10}(P_{\ge N}u_0) \mbox{ with }\|v_0\|_{\dot H^{s_c}(\R^d)}\le \delta_0,\quad \mbox{and }\quad w_0\in L^2(\R^d).
$$
Now, let $v$ be the solution of the following \emph{time cut-off} equation,
 \begin{equation*}
   \left\{ \aligned
    &i\partial_{t}v+\Delta v= \chi_{\le 1}(t)|v|^p v,
    \\
    &v(0,x)  =v_0(x).
   \endaligned
  \right.
 \end{equation*}
In this step, we prove that the analogous estimates in Step 1 hold true for the nonlinear solution $v$. That is,
\begin{align*}
\Big\|\langle t^{\alpha_0}|\nabla|\rangle^{\beta_0}|\nabla|^{s_c}P_{M}v\Big\|_{l^\infty_M L^2_t L^{\frac{2d}{d-2}}_x(\{M\ge 1\}\times\R\times\R^d)}
\lesssim \|v_0\|_{\dot H^{s_c}(\R^d)},
\end{align*}
which we use later with $t\gtrsim 1$.

In step 3, we prove the uniform in time boundedness of $\|w(t)\|_{L^2_x(\R^d)}$.   Note that $w$ obeys the equation of
\begin{equation*}
    i\partial_{t}w+\Delta w= |u|^pu-\chi_{\le 1}(t)|v|^p v.
 \end{equation*}
 We find that the nonlinearity obeys
 \begin{align*}
\big||u|^pu-\chi_{\le 1}(t)|v|^pv\big|
\lesssim \big(|u|^p+|\chi_{\lesssim 1}(t)v|^p\big)\big(|w|+|\chi_{\gtrsim 1}(t)v|\big).
\end{align*}
Due to the good estimates on $\chi_{\gtrsim 1}(t)v$ obtained in Step 2, we can prove the desired estimate by the almost mass conservation of $w$.

\section{Preliminary}

\subsection{Notation}

We write $X \lesssim Y$ or $Y \gtrsim X$ to indicate $X \leq CY$ for some constant $C>0$. If $C$ depends upon some additional
parameters, we will indicate this with subscripts; for example, $X\lesssim_a Y$ denotes the
assertion that $X\le C(a)Y$ for some $C(a)$ depending on $a$. We use $O(Y)$ to denote any quantity $X$
such that $|X| \lesssim Y$.  We use the notation $X \sim Y$ whenever $X \lesssim Y \lesssim X$. Moreover, we use the notation $X\ll  Y$ to  indicate $X \leq C^{-1}Y$.  

The notation 
$|\nabla|^\alpha=(-\partial^2_x)^{\alpha/2}$.
We denote $\mathcal S(\R^d)$ to be the Schwartz Space in $\R^d$, and $\mathcal S'(\R^d)$ to be the topological dual of $\mathcal S(\R^d)$.  Let $h\in \mathcal S'(\R^{d+1})$, we use
$\|h\|_{L^q_tL^p_x}$ to denote the mixed norm
$\Big(\displaystyle\int\|h(\cdot,t)\|_{L^p}^q\
dt\Big)^{\frac{1}{q}}$, and $\|h\|_{L^q_{xt}}:=\|h\|_{L^q_xL^q_t}$. Sometimes, we use the notation $q'=\frac{q}{q-1}$.

Throughout this paper, we use $\chi_{\le a}$ for $a\in \R^+$ to be the smooth function
\begin{align*}
\chi_{\le a}(x)=\left\{ \aligned
1, \ & |x|\le a,\\
0,    \ &|x|\ge \frac{11}{10} a.
\endaligned
  \right.
\end{align*}
Moreover, we denote $\chi_{\ge a}=1-\chi_{\le a}$ and $\chi_{a\le \cdot\le b}=\chi_{\le b}-\chi_{\le a}$. We denote $\chi_{a}=\chi_{\le 2a}-\chi_{\le a}$ for short.
%


Also, we need some Fourier operators. First, we recall the Fourier transform and its inverse formula.
We denote the Fourier transform  by $\hat f$ or $\mathscr{F} f$ as
\begin{align*}
\big( \,\mathscr F f(\xi)\,\,\mbox{or}\,\,\big)\hat f(\xi) &=\int_{\R^d} e^{-2\pi i x\cdot \xi}f(x)\,dx,
\end{align*}
and its inverse transform  by $\check f$ or $\mathscr{F}^{-1} f$ as
\begin{align*}
\big( \,\mathscr F^{-1} f(x)\,\,\mbox{or}\,\,\big)\check f(x) &=\int_{\R^d} e^{2\pi i x\cdot \xi}f(\xi)\,d\xi.
\end{align*}

For each number $N > 0$, we define the Fourier multipliers $P_{\le N}, P_{> N}, P_N$ as
\begin{align*}
\widehat{P_{\leq N} f}(\xi) &:= \chi_{\leq N}(\xi) \hat f(\xi),\\
\widehat{P_{> N} f}(\xi) &:= \chi_{> N}(\xi) \hat f(\xi),\\
\widehat{P_N f}(\xi) &:= \chi_{N}(\xi) \hat
f(\xi),
\end{align*}
and similarly $P_{<N}$ and $P_{\geq N}$.  
 We will usually use these multipliers when $N$ are \emph{dyadic numbers} (that is, of the form $2^k$
for some integer $k$).

Moreover,  we also need the notations
$$
\nabla_{\xi}=\{\partial_{\xi_1},\cdots,\partial_{\xi_d}\};\quad
\partial_\xi^{l}=\partial_{\xi_1}^{l^1}\cdots\partial_{\xi_d}^{l^d}, \mbox{ for any } l=\{l^1,\cdots,l^d\}\in \R^{d}.
$$

\subsection{Some basic lemmas}
First, we need the following radial Sobolev embedding, see \cite{TaViZh} for example.
\begin{lem}\label{lem:radial-Sob}
Let $\alpha,q,p,s$ be the parameters which satisfy
$$
\alpha>-\frac dq;\quad \frac1q\le \frac1p\le \frac1q+s;\quad 1\le p,q\le \infty; \quad 0<s<d
$$
with
$$
\alpha+s=d(\frac1p-\frac1q).
$$
Moreover, at most one of the equalities hold:
$$
p=1,\quad p=\infty,\quad q=1,\quad q=\infty,\quad \frac1p=\frac1q+s.
$$
Then
\begin{align*}
\big\||x|^\alpha u\big\|_{L^q(\R^d)}\lesssim \big\||\nabla|^su\big\|_{L^p(\R^d)}.
\end{align*}
\end{lem}

The second is the following fractional Leibniz rule, see \cite{KePoVe-CPAM-1993, BoLi-KatoPonce, Li-KatoPonce} and the references therein.
\begin{lem}\label{lem:Frac_Leibniz}
Let $0<s<1$, $\frac12<p\le \infty$, and $1<p_1,p_2,p_3, p_4 \le \infty$ with $\frac1p=\frac1{p_1}+\frac1{p_2}$, $\frac1p=\frac1{p_3}+\frac1{p_4}$, and let $f,g\in \mathcal S(\R^d)$,  then
\begin{align*}
\big\||\nabla|^s(fg)\big\|_{L^p}\lesssim \big\||\nabla|^sf\big\|_{L^{p_1}}\|g\|_{L^{p_2}}+ \big\||\nabla|^sg\big\|_{L^{p_3}}\|f\|_{L^{p_4}}.
\end{align*}
\end{lem}

A simple consequence is the  following elementary inequality.
\begin{lem}\label{lem:frac_Hs}
For any $a>0, 1\le p\le \infty, 0\le \gamma<\frac dp$, and $ |\nabla|^\gamma g\in L^p(\R^d)$,
\begin{align}
\big\||\nabla|^\gamma\big(\chi_{\le a}g\big)\big\|_{L^p(\R^d)}\lesssim \big\||\nabla|^\gamma g\big\|_{L^p(\R^d)}. \label{15.37}
\end{align}
Here the implicit constant is independent on $a$.
The same estimate holds for $\chi_{\ge a}g$.
\end{lem}
\begin{proof}
The case $\gamma=0$ is trivial.
Further, we may assume that $0<\gamma<1$. Otherwise, we can use the standard Leibniz rule and the H\"older inequality to reduce the derivatives.

From Lemma \ref{lem:Frac_Leibniz}, the H\"older and Sobolev inequalities, we have
\begin{align*}
\big\||\nabla|^\gamma\big(\chi_{\le a}g\big)\big\|_{L^p(\R^d)}
\lesssim &  \big\||\nabla|^\gamma\chi_{\le a}\big\|_{L^\frac{d}{\gamma}(\R^d)}\big\|g\big\|_{L^\frac{dp}{d-p\gamma}(\R^d)}
+ \big\|\chi_{\le a}\big\|_{L^\infty(\R^d)}\big\||\nabla|^\gamma g\big\|_{L^p(\R^d)}\\
\lesssim &  \Big(\big\||\nabla|^\gamma\chi_{\le a}\big\|_{L^\frac{d}{\gamma}(\R^d)}+\big\|\chi_{\le a}\big\|_{L^\infty(\R^d)}\Big)\big\||\nabla|^\gamma g\big\|_{L^p(\R^d)}.
\end{align*}
Note that $\|\chi_{\le a}\|_{L^\infty(\R^d)}\lesssim1$ and
\begin{align*}
\big\||\nabla|^\gamma\chi_{\le a}\big\|_{L^\frac{d}{\gamma}(\R^d)}=&a^{-\gamma}\big\||\nabla|^\gamma\chi_{\le 1}(\frac {\cdot}a)\big\|_{L^\frac{d}{\gamma}(\R^d)}
=\big\||\nabla|^\gamma\chi_{\le 1}\big\|_{L^\frac{d}{\gamma}(\R^d)}\lesssim1.
\end{align*}
Hence we obtain \eqref{15.37}.

Note that
$$
\chi_{\ge a}g=1-\chi_{\le a}g,
$$
then by \eqref{15.37}, we have
\begin{align*}
\big\|\chi_{\ge a}g\big\|_{\dot H^\gamma(\R^d)}
\lesssim
\big\|g\big\|_{\dot H^\gamma(\R^d)}+\big\|\chi_{\le a}g\big\|_{\dot H^\gamma(\R^d)}
\lesssim
\big\|g\big\|_{\dot H^\gamma(\R^d)}.
\end{align*}
Hence, the same estimate holds for $\chi_{\le a}g$. Thus we finish the proof of the lemma.
\end{proof}

Note that the condition $ \gamma<\frac dp$ in Lemma \ref{lem:frac_Hs} can be removed if $a\gtrsim 1$. We also need the following Littlewood-Paley inequality, see for example \cite{Tr-book-92}. 
\begin{lem}\label{lem:LittlePaley-ineq}
Let $p\in (1,+\infty)$ and $f\in L^p(\R^d)$. 
Then
\begin{align*}
\big\|f\big\|_{L^p(\R^d)}\sim \big\|P_{\le 1}f\big\|_{L^p(\R^d)} + \Big\|\Big(\sum\limits_{j=1}^\infty |P_j f|^2\Big)^\frac12\Big\|_{L^p(\R^d)}.
\end{align*}
\end{lem}

Moreover, we need the following mismatch result, which is  helpful in commuting the spatial and the frequency cutoffs.
\begin{lem}[Mismatch estimates, see \cite{LiZh-APDE}]\label{lem:mismatch}
Let $\phi_1$ and $\phi_2$ be smooth functions obeying
$$
|\phi_j| \leq 1 \quad \mbox{ and }\quad \mbox{dist}(\emph{supp}
\phi_1,\, \emph{supp} \phi_2 ) \geq A,
$$
for some large constant $A$.  Then for $\sigma>0$, $M\le 1$ and $1\leq r\leq
q\leq \infty$,
\begin{align}
\bigl\| \phi_1 |\nabla|^\sigma P_{\leq M} (\phi_2 f)
\bigr\|_{L^q_x(\R^d)}
    &+ \bigl\| \phi_1 \nabla |\nabla|^{\sigma-1} P_{\leq M} (\phi_2 f) \bigr\|_{L^q_x(\R^d)}
    \lesssim A^{-\sigma-\frac dr + \frac dq} \|\phi_2 f\|_{L^r_x(\R^d)};\label{eqs:lem-mismath-1}\\
    \bigl\| \phi_1  \nabla P_{\leq M} (\phi_2 f)
\bigr\|_{L^q_x(\R^d)}&\lesssim_m M^{1-m}A^{-m}\|f\|_{L^q_x(\R^d)}, \mbox{
for any } m\ge0.\label{eqs:lem-mismath-2}
\end{align}
\end{lem}

Furthermore, we need the following  elementary formulas. The first one is 

\begin{lem}\label{lem:muli-Lei-formula}
Let the scale function $g\in \mathcal S(\R^d)$, and the phase function $\phi\in (\mathcal S(\R^d))$ with 
$$
\inf\limits_{y\in\R^d}|\nabla \phi(y)|>0,
$$
then for any integer $N$,
\begin{equation}\label{for:phs-inbyparts}
\int_{\R^d} e^{i\phi(y)} g(y)\,dy=
\int_{\R^d}e^{i\phi(y)}\>\nabla_y\cdot\Big(\frac{\nabla_y\phi}{i|\nabla_y\phi|^2}\nabla_y\Big)^{N-1}\cdot\Big(\frac{\nabla_y\phi}{i|\nabla_y\phi|^2}g(y)\Big)\,dy.
\end{equation}
\end{lem}
\begin{proof}
Note that 
\begin{align*}
e^{i\phi}=\nabla_y e^{i\phi}\cdot\frac{\nabla_y\phi}{i|\nabla_y\phi|^2}.
\end{align*}
Then we write 
\begin{equation*}
\int_{\R^d} e^{i\phi(y)} g(y)\,dy=
\int_{\R^d} \nabla_y e^{i\phi}\cdot\frac{\nabla_y\phi}{i|\nabla_y\phi|^2} g(y)\,dy.
\end{equation*}
By integration-by-parts, we get 
\begin{equation*}
\int_{\R^d} e^{i\phi(y)} g(y)\,dy=
\int_{\R^d} e^{i\phi}\> \nabla_y\cdot\left(\frac{\nabla_y\phi}{i|\nabla_y\phi|^2} g(y)\right)\,dy.
\end{equation*}
This obtains the case of $N=1$. 
Denote that
$$
\Lambda_N(g)=\nabla_y\cdot\Big(\frac{\nabla_y\phi}{i|\nabla_y\phi|^2}\nabla_y\Big)^{N-1}\cdot\Big(\frac{\nabla_y\phi}{i|\nabla_y\phi|^2}g(y)\Big),
$$
then we have
$$
\Lambda_{N+1}(g)=\nabla_y\cdot\left(\frac{\nabla_y\phi}{i|\nabla_y\phi|^2} \Lambda_N(g)\right).
$$
The identity is then followed from the induction.
\end{proof}
The second one is 
\begin{lem}\label{lem:muli-Lei-formula}
Let the vector function $f\in (\mathcal S(\R^d))^d$ and the scale function $g\in \mathcal S(\R^d)$, then for any integer $N$,
\begin{equation*}
\nabla_{\xi}\cdot \big(f\>\nabla_{\xi} \big)^{N-1}\cdot
(fg)=\sum\limits_{\begin{subarray}{c}
l_1,\cdots,l_N\in\R^d,l'\in\R^d;\\
|l_j|\le
j;|l_1|+\cdots+|l_N|+|l'|=N
\end{subarray}}
C_{l_1,\cdots,l_N,l'}\partial_\xi^{l_1}f\cdots
\partial_\xi^{l_N}f\>\partial_\xi^{l'}g.
\end{equation*}
\end{lem}
\begin{proof}
When $N=1$, it is directly followed from the Leibniz rule.
Denote that
$$
\Lambda_N(f,g)=\nabla_{\xi}\cdot \big(f\>\nabla_{\xi} \big)^{N-1}\cd
(fg),
$$
then we have
$$
\Lambda_N(f,g)= \nabla_{\xi}\cdot \big(f\Lambda_{N-1}(f,g)\big).
$$
The identity is then followed from the induction.
\end{proof}

\subsection{Linear Schr\"odinger operator}

Let the operator $S(t)=e^{it\Delta}$ be the linear Schr\"odinger flow, that is,
$$
(i\partial_t+\Delta)S(t)\equiv 0.
$$
The following are some fundamental properties of the operator $e^{it\Delta}$. The first is the explicit formula, see  for example Cazenave \cite{Cazenave-book}.
\begin{lem}\label{lem:formula-St}
For all $\phi\in \mathcal S(\R^d)$, $t\neq 0$,
$$
S(t)\phi(x)=\frac{1}{(4\pi it)^{\frac d2}}\int_{\R^d} e^{\frac{i|x-y|^2}{4t}}\phi(y)\,dy.
$$
Moreover, for any $r\ge2$,
$$
\|S(t)\phi\|_{L^r_x(\R^d)}\lesssim |t|^{-d(\frac12-\frac1r)}\|\phi\|_{L^{r'}(\R^d)}.
$$
\end{lem}

The following is the  standard Strichartz estimates, see for example \cite{KeTa-Strichartz}.
\begin{lem}\label{lem:strichartz}
Let $I$ be a compact time interval and let $u: I\times\R^d \to
\mathbb \R$ be a solution to the inhomogeneous Schr\"odinger equation
$$
iu_{t}- \Delta u + F = 0.
$$
Then for any $t_0\in I$,  any pairs $(q_j,r_j), j=1,2$ satisfying
$$
 q_j\ge 2, \,\, r_j\ge 2,\,\,  \mbox{and }\,\, \frac2{q_j}+\frac d{r_j}=\frac d2,
$$
the following estimates hold,
\begin{align*}
\bigl\|u\bigr\|_{C(I;L^2(\R^d))}+\big\|u\big\|_{L^{q_1}_tL^{r_1}_x(I\times\R^d)}
\lesssim \bigl\|u(t_0)\bigr\|_{L^2_x(\R^d)}+
\bigl\|F\bigr\|_{L^{q_2'}_tL^{r_2'}_x(I\times\R^d)}.
\end{align*}
\end{lem}

We also need the special Strichartz estimates for radial data, which was firstly proved by Shao \cite{Shao}, and then developed in \cite{CL,GW}.
\begin{lem}[Radial Strichartz estimates]\label{lem:radial-Str}
Let $g\in L^2(\R^d)$ be a radial function,  then for any triple $(q,r,\gamma)$ satisfying
\begin{align}
\gamma\in\R, \,\, q\ge 2, \,\, r> 2,\,\, \frac2q+\frac {2d-1}{r}<\frac{2d-1}2,\,\, \mbox{and }\,\, \frac2q+\frac d{r}=\frac d2+\gamma,\label{Str-conditions}
\end{align}
we have that
$$
\big\||\nabla|^{\gamma} e^{it\Delta}g\big\|_{L^q_{t}L^r_x(\R\times\R^d)}\lesssim  \big\|g\big\|_{L^2(\R^d)}.
$$
\end{lem}
Followed by the standard TT*-method, a direct consequence of the lemma above is
\begin{cor}[Inhomogeneous radial Strichartz estimates]\label{cor:radial-Str-inhomo}
Let $F\in L^{\tilde q'}_{t}L^{\tilde r'}_x(\R^{d+1})$ be a radial function in $x$, then
$$
\Big\|\int_0^t e^{i(t-s)\Delta} |\nabla|^\gamma F(s)\,ds\Big\|_{L^q_{t}L^r_x(\R^{1+d})}\lesssim \big\||\nabla|^{-\tilde \gamma} F\big\|_{L^{\tilde q'}_{t}L^{\tilde r'}_x(\R^{1+d})},
$$
where the triples  $(q,r,\gamma)$,   $(\tilde q,\tilde r,\tilde\gamma)$ satisfy \eqref{Str-conditions}.
\end{cor}
The following is a remark regarding Lemma \ref{lem:radial-Str} and Corollary \ref{cor:radial-Str-inhomo}.
 \begin{remark}\label{rem:derivatives-radial}
One may ask about  the optimal smoothing effect one can gain from the radial Strichartz estimates, corresponding to the supremum of $\gamma$ as above.  In fact, from Lemma \ref{lem:radial-Str},   we find that
$$
\gamma<\frac2q\cdot \frac{d-1}{2d-1}.
$$
An equivalent inequality is 
$$
\gamma<(d-1)\left(\frac12-\frac1r\right).
$$
Hence, $\frac2q\cdot \frac{d-1}{2d-1}-$ is the most derivatives we can gain from the radial Strichartz estimates when we fix $q$; 
while $(d-1)(\frac12-\frac1r)-$ is the most derivatives we can gain from the radial Strichartz estimates when we fix $r$.  
In particular, when $q=2$, we denote the triples   
$$
(q_0,r_0,\gamma_0):=\left(2, \frac{2(2d-1)}{2d-3}+, \frac{d-1}{2d-1}-\right).
$$
Therefore, $\gamma_0=\frac{d-1}{2d-1}-$  is the most derivatives we can gain from the radial Strichartz estimates. While we can gain double derivatives from the inhomogeneous radial Strichartz estimates by setting $\tilde\gamma=\gamma$ in Corollary \ref{cor:radial-Str-inhomo}.

Further, we note that from Corollary \ref{cor:radial-Str-inhomo}, for any $h\in L^2_t L^{r_0'}_x(\R\times \R^d)$,
\begin{align*}
\Big\|\int_0^t e^{i(t-s)\Delta} |\nabla|^{\gamma_0}  h(s)\,ds\Big\|_{L^2_{t}L^{\frac{2d}{d-2}}_x(\R\times\R^{d})}
\lesssim 
\big\|h\big\|_{L^2_{t}L^{r_0'}_x(\R\times\R^{d})}.
\end{align*}
Further, by Sobolev's inequality, it infers that  
\begin{align}\label{radial-Str-inhomo-conseq}
\Big\|\int_0^t e^{i(t-s)\Delta}  h(s)\,ds\Big\|_{L^2_{t}L^{\rho_0}_x(\R\times\R^{d})}
\lesssim 
\big\|h\big\|_{L^2_{t}L^{r_0'}_x(\R\times\R^{d})},
\end{align}
where $\rho_0$ satisfies that 
$$
\frac1{\rho_0}=\frac{d-2}{2d}-\frac{\gamma_0}{d}=\frac12-\frac{3d-2}{d(2d-1)}+.
$$
\end{remark}

Now we need the following specific truncated inhomogeneous  radial Strichartz estimates. 
\begin{lem}\label{lem:trun-inho-Str}
Let $T\ge 2, r\ge \rho_0$, and let the triples $(\tilde q,\tilde r,\tilde \gamma)$ satisfy that for some $\theta\in [0,1]$, 
$$
\frac{1}{\tilde q}\le (1-\theta)+\frac12\theta;\quad 
\frac{1}{\tilde r}=\frac{1-\theta}{r'}+\frac\theta{r_0'};\quad
\tilde\gamma=\theta \Big(\frac{d-2}{2}-\gamma_0-\frac{d}r\Big).
$$ 
Suppose that $F$ is a radial function in $x$ such that $|\nabla|^{\tilde \gamma}F\in L^{\tilde q}_{t}L^{\tilde r}_x([0,2]\times\R^{d})$, then
\begin{align*}
\Big\|\int_0^t e^{i(t-s)\Delta}  \chi_{\le 1}(s) F(s)\,ds\Big\|_{L^1_{t}L^{r}_x([T,2T]\times\R^{d})}\lesssim T^{[-d(\frac12-\frac1r)+1](1-\theta)+\frac12\theta}\big\||\nabla|^{\tilde\gamma} F\big\|_{L^{\tilde q}_{t}L^{\tilde r}_x([0,2]\times\R^{d})}.
\end{align*}
\end{lem}
\begin{proof}
By Lemma \ref{lem:formula-St}, 
\begin{align*}
\Big\|\int_0^t e^{i(t-s)\Delta} \chi_{\le 1}(s) F(s)\,ds\Big\|_{L^r_x(\R^{d})}
\lesssim 
\int_0^t |t-s|^{-d(\frac12-\frac1r)} \chi_{\le 1}(s) \big\|F(s)\big\|_{L^{r'}_x(\R^{d})}\,ds.
\end{align*}
Since $t-s\sim t$ when $t\ge 2, s\le \frac{11}{10}$, we further have that 
\begin{align*}
\Big\|\int_0^t e^{i(t-s)\Delta} \chi_{\le 1}(s) F(s)\,ds\Big\|_{L^r_x(\R^{d})}
\lesssim 
|t|^{-d(\frac12-\frac1r)}\big\|F\big\|_{L^1_tL^{r'}_x([0,2]\times \R^{d})}.
\end{align*}
This gives that 
\begin{align}\label{Inh-1}
\Big\|\int_0^t e^{i(t-s)\Delta} \chi_{\le 1}(s) F(s)\,ds\Big\|_{L^1_{t}L^{r}_x([T,2T]\times\R^{d})}
\lesssim 
T^{-d(\frac12-\frac1r)+1}\big\|F\big\|_{L^1_tL^{r'}_x([0,2]\times \R^{d})}.
\end{align}

On the other hand,  
applying the Sobolev embedding and the inequality \eqref{radial-Str-inhomo-conseq}, we obtain that for any $r\ge \rho_0$, 
\begin{align*}
\Big\|\int_0^t e^{i(t-s)\Delta} \chi_{\le 1}(s) F(s)\,ds\Big\|_{L^2_{t}L^{r}_x(\R\times\R^{d})}
\lesssim &
\Big\|\int_0^t e^{i(t-s)\Delta} \chi_{\le 1}(s) |\nabla|^{d(\frac1{\rho_0}-\frac1r)} F(s)\,ds\Big\|_{L^2_{t}L^{\rho_0}_x(\R\times\R^{d})}\\
\lesssim &
\big\|\chi_{\le 1}(t) |\nabla|^{d(\frac1{\rho_0}-\frac1r)} F(t)\big\|_{L^2_{t}L^{r_0'}_x(\R\times\R^{d})}.
\end{align*}
Therefore, by the time support, we get that 
\begin{align}\label{Inh-2}
\Big\|\int_0^t e^{i(t-s)\Delta} \chi_{\le 1}(s) F(s)\,ds\Big\|_{L^1_{t}L^{r}_x([T,2T]\times\R^{d})}
\lesssim &
T^\frac12\big\| |\nabla|^{d(\frac1{\rho_0}-\frac1r)} F\big\|_{L^2_{t}L^{r_0'}_x([0,2]\times\R^{d})}.
\end{align}

Then by interpolation between \eqref{Inh-1}, \eqref{Inh-2} and  H\"older's inequality in time, we obtain the desired estimates.
\end{proof}

\begin{remark}\label{rem:trun-inho-Stri}
 Note that if 
\begin{align}\label{theta-cond}
(-d(\frac12-\frac1r)+1)(1-\theta)+\frac12\theta<0,
\end{align}
then by dyadic decomposition in time, we have that 
\begin{align}\label{est-trun-inho-Stri}
\Big\|\int_0^t e^{i(t-s)\Delta}  \chi_{\le 1}(s) F(s)\,ds\Big\|_{L^1_{t}L^{r}_x([2,+\infty)\times\R^{d})}
\lesssim 
\big\||\nabla|^{\tilde\gamma} F\big\|_{L^{\tilde q}_{t}L^{\tilde r}_x([0,2]\times\R^{d})}.
\end{align}
\end{remark}

\section{Linear flow estimates on localized functions}
As the first part, we begin with the linear flow estimates.
In this section, we give the following estimates.
\begin{prop}\label{prop:lineares-cpt-largetime-I}
Let  $2\le r<\infty$, then for any  $t: |t|\ge \frac {100}N$, and any $s$ satisfying
$$
0\le s<(d-2)\big(\frac12-\frac1r\big)+s_c,
$$
the following estimate holds,
\begin{align}
\Big\||\nabla|^{s }\big(e^{it\Delta}\big(\chi_{\le 10}P_{\ge N}g\big)\big)\Big\|_{L^r_x(\R^d)}
\lesssim &
N^{-(d-2)\big(\frac12-\frac1r\big)+s-s_c}|t|^{-(d-1)\big(\frac12-\frac1r\big)}\|P_{\ge N}g\|_{\dot H^{s_c}(\R^d)}.\label{13.26}
\end{align}
Moreover, let $d\ge3$,  $(q,r,\gamma)$ be the triple satisfying \eqref{Str-conditions} and
\begin{align}
2\le q <\infty, \quad \frac1q< (d-1)\big(\frac12-\frac1r\big).
\label{q-r-gamma-condition}
\end{align}
Then there exist $s_*=s_*(q)<0, \eta_*=\eta_*(q,r,\gamma)> 0$, such that for any $\alpha\ge 1, \beta>0$ with $\alpha\beta\le \eta_*$, and any $s_c\in [s_*,0)$,
the following estimate holds,
\begin{align}
\Big\|\langle t^{\alpha}|\nabla|\rangle^{\beta}|\nabla|^{s_c+\gamma}\big(e^{it\Delta}\chi_{\le 10}(P_{\ge N}g)\big)\Big\|_{L^q_tL^r_x(\R\times\R^d)}
\lesssim \|P_{\ge N}g\|_{\dot H^{s_c}(\R^d)}.\label{14.49}
\end{align}
\end{prop}
 \begin{remark}
From the proof of Proposition \ref{prop:lineares-cpt-largetime-I} below, it also follows that $s_*(q)$ can be chosen to be a increasing function with $s_*(+\infty)=0$. What we use below is $s_*(2)$.
\end{remark}

The proof of the proposition is based on the following two lemmas and corollaries. First of all, we show  the estimate in the local domain.
\begin{lem}\label{lem:local-fre-I}
Let $M\ge1$, $r\ge2$ and $s\ge 0$, then for any  $t: |t|\ge \frac {100}M$, any $K\in \Z^+$,
\begin{align}
\Big\||\nabla|^{s }\big(\chi_{\le \frac1{10}M|t|}e^{it\Delta}\big(\chi_{\le 10}P_Mg\big)\big)\Big\|_{L^r_x(\R^d)}
\lesssim_{s ,K} &
|t|^{-d(\frac 12-\frac1r)}M^{-K}\|P_Mg\|_{\dot H^{s_c}(\R^d)}.\label{16.17}
\end{align}
\end{lem}
\begin{proof}
First, we show that for any $M\ge 1$, $t: |t|\ge \frac {100}M$, $s\in \Z^+\cup\{0\}$ and $K\in \Z^+$,
\begin{align}
\Big|| \nabla|^s \Big(\chi_{\le \frac{1}{10}M|t|}\>e^{it\Delta}\big(\chi_{\le 10}P_M g\big)\Big)\Big|
\lesssim_K &
|t|^{-\frac d2}M^{-K}\|P_Mg\|_{\dot H^{s_c}(\R^d)}. \label{16.16}
\end{align}
To show this, we use the formulas in Lemma \ref{lem:formula-St}  and the inverse Fourier transform to obtain
\begin{align}
e^{it\Delta}\big(\chi_{\le 10}P_M g\big)(x) 
=&\frac{1}{(4\pi it)^\frac d2}\int_{\R^d} e^{\frac{i|x-y|^2}{4t}}\chi_{\le 10}(y)P_M  g(y)\,dy\notag\\
=&\frac{1}{(4\pi it)^\frac d2}\int_{\R^d}\!\!\int_{\R^d} e^{\frac{i|x-y|^2}{4t}+2\pi iy\cdot \xi}\chi_{\le 10}(y)\chi_M (\xi)\hat g(\xi)\,dyd\xi.
\label{eq:15.49}
\end{align}
Fix $x,\xi$, and define the phase as
$$
\phi(y)=-\frac{x\cdot y}{2t}+\frac{|y|^2}{4t}+y\cdot \xi,
$$
then from \eqref{eq:15.49},
\begin{align}
e^{it\Delta}\big(\chi_{\le 10}P_M g\big)(x) =\frac{e^{\frac{i|x|^2}{4t}}}{(4\pi it)^\frac d2}\int_{\R^d}\!\!\int_{\R^d} e^{i\phi(y)}\chi_{\le 10}(y)\chi_M (\xi)\hat g(\xi)\,dyd\xi.
\label{eq:1.16}
\end{align}
Moreover, we have
\begin{align}
\nabla_y \phi(y)=\frac{y-x}{2t}+\xi, \label{15.08-I}
\end{align}
and
\begin{align}
\partial_{y_jy_k}\phi(y)=\frac{\delta_{jk}}{2t},\quad
\partial_{y_jy_ky_h}\phi(y)=0, \mbox{ for any}\quad  j,k,h\in\{1,\cdots,d\}. \label{10.52-I}
\end{align}
Note that
$$
|t|M\ge 100, \,\, |\xi|\ge \frac{9}{10}M, \,\, |y|\le 11,\,\,  \mbox{ and } |x|\le \frac{11}{100}M|t|,
$$
from \eqref{15.08-I} we have
\begin{align}
|\nabla \phi(y)|\gtrsim |\xi|.\label{2.04}
\end{align}
Then using the formula \eqref{for:phs-inbyparts} to the right-hand side of \eqref{eq:1.16}, we obtain
\begin{align}
\chi_{\le \frac{1}{10}M|t|}(x)&\>e^{it\Delta}\big(\chi_{\le 10}P_M g\big)(x)
=
\chi_{\le \frac{1}{10}M|t|}(x)\>\frac{C_Ke^{\frac{i|x|^2}{4t}}}{t^\frac d2}\int_{\R^d}\!\!\int_{\R^d}e^{i\phi(y)}\notag\\
 &\cdot\nabla_y\cdot\Big(\frac{\nabla_y\phi}{i|\nabla_y\phi|^2}\nabla_y\Big)^{K-1}\cdot\Big(\frac{\nabla_y\phi}{i|\nabla_y\phi|^2}\chi_{\le 10}(y)\Big)\,dy\cdot\chi_M (\xi)\hat g(\xi)\,d\xi. \label{10.03-I}
\end{align}
We claim that
\begin{align}
\Big|\nabla_y\cdot\Big(\frac{\nabla_y\phi}{i|\nabla_y\phi|^2}\nabla_y\Big)^{K-1}\cdot\Big(\frac{\nabla_y\phi}{i|\nabla_y\phi|^2}\chi_{\le 10}(y)\Big)\Big|
\lesssim |\xi|^{-K}\chi_{\lesssim 1}(\cdot).\label{10.36-I}
\end{align}
Indeed, from Lemma \ref{lem:muli-Lei-formula}, we expand the left-hand side of \eqref{10.36-I} as
\begin{equation}\label{eq:18.44-I}
\sum\limits_{\begin{subarray}{c}
l_1,\cdots,l_K\in\R^d,l'\in\R^d;\\
|l_j|\le
j;|l_1|+\cdots+|l_K|+|l'|=K
\end{subarray}}
C_{l_1,\cdots,l_K,l'}\partial_y^{l_1}\Big(\frac{\nabla_y\phi}{i|\nabla_y\phi|^2}\Big)\cdots
\partial_y^{l_K}\Big(\frac{\nabla_y\phi}{i|\nabla_y\phi|^2}\Big)\>\partial_y^{l'}\big(\chi_{\le 10}(\cdot)\big).
\end{equation}
Note that from \eqref{10.52-I} and \eqref{2.04}, we have that for any non-negative integer vectors  $l$, $l'$,
$$
\Big|\partial_y^{l}\Big(\frac{\nabla_y\phi}{i|\nabla_y\phi|^2}\Big)\Big|\lesssim \frac{1}{|\xi|}\frac{1}{|t\xi|^{|l|}},
$$
and
$$
\Big|\partial_y^{l'}\big(\chi_{\le 10}(\cdot)\big)\Big|\lesssim \chi_{\lesssim 1}(\cdot).
$$
Hence, using these two estimates, and noting that $|t\xi|\gtrsim 1$, we have that 
\begin{align*}
|\eqref{eq:18.44-I}|
\lesssim & \sum\limits_{\begin{subarray}{c}
l_1,\cdots,l_K\in\R^d,l'\in\R^d;\\
|l_j|\le
j;|l_1|+\cdots+|l_K|+|l'|=K
\end{subarray}}
\frac{1}{|\xi|}\frac{1}{|t\xi|^{|l_1|}}\cdots \frac{1}{|\xi|}\frac{1}{|t\xi|^{|l_K|}}\cdot\chi_{\lesssim 1}(y)\\
= &\sum\limits_{\begin{subarray}{c}
l_1,\cdots,l_K\in\R^d,l'\in\R^d;\\
|l_j|\le
j;|l_1|+\cdots+|l_K|+|l'|=K
\end{subarray}}
 |\xi|^{-K} |t\xi|^{-(|l_1|+\cdots +|l_K|)}\cdot\chi_{\lesssim 1}(y)\\
\lesssim & |\xi|^{-K}\chi_{\lesssim 1}(y).
\end{align*}
Therefore, we obtain \eqref{10.36-I}.

Inserting \eqref{10.36-I} into \eqref{10.03-I}, we obtain
\begin{align*}
\big|\chi_{\le \frac{1}{10}M|t|}\>e^{it\Delta}\big(\chi_{\le 10}P_M g\big)\big|
\lesssim  &
|t|^{-\frac d2}\int_{\R^d}\!\!\int_{\R^d} \chi_{\lesssim 1}(y)|\xi|^{-K}\chi_M (\xi)\hat g(\xi)\,dy\,d\xi\\
\lesssim &
|t|^{-\frac d2}M^{-K-s_c+d}\|P_Mg\|_{\dot H^{s_c}(\R^d)}.
\end{align*}
Note that when the derivatives hit the cut-off functions $\chi_{\le \frac{1}{10}M|t|}$ and $\chi_{\le 10}$, the estimates on
$$\chi_{\le \frac{1}{10}M|t|}\>e^{it\Delta}\big(\chi_{\le 10}P_M g\big)$$
 become better, hence by choosing $K$ suitable large, we obtain that for any $s\in \Z^+$,
\begin{align*}
\Big|| \nabla|^s \Big(\chi_{\le \frac{1}{10}M|t|}\>e^{it\Delta}\big(\chi_{\le 10}P_M g\big)\Big)\Big|
\lesssim  &
|t|^{-\frac d2}\int_{\R^d}\!\!\int_{\R^d} \chi_{\lesssim 1}(y)|\xi|^{-K}\chi_M (\xi)\hat g(\xi)\,dy\,d\xi\\
\lesssim &
|t|^{-\frac d2}M^{-K-s_c+d}\|P_Mg\|_{\dot H^{s_c}(\R^d)}.
\end{align*}
Replacing $K$ by $K-s_c+d$, we obtain \eqref{16.16}.

Further, using \eqref{16.16}, Lemma \ref{lem:mismatch}, H\"older's inequality and interpolation when $s$ is not an integer, we obtain  \eqref{16.17}.
\end{proof}
Then a direct consequence of the previous lemma is the following corollary. 
\begin{cor}\label{cor:local-fre-I}
Let $(q,r,\gamma)$  be the triple satisfying the same conditions as in Proposition \ref{prop:lineares-cpt-largetime-I}, and let  $\alpha, \beta$ be the constants satisfying
\begin{align}
\alpha\ge 1,\quad  \beta> 0,\qquad d(\frac12-\frac 1r)> \max\{\alpha\beta,-\alpha(s_c+\gamma)\} +\frac1q.\label{17.33-I}
\end{align}
Then there exists $s_{*,1}=s_{*,1}(q)<0$ such that for any $s_c\in [s_{*,1},0)$, the following estimate holds,
\begin{align}
\Big\|&\langle t^{\alpha}|\nabla|\rangle^{\beta}|\nabla|^{s_c+\gamma}\big(\chi_{\le \frac1{10}M|t|}e^{it\Delta}\big(\chi_{\le 10}P_Mg\big)\big)\Big\|_{L^q_tL^r_x(\R\times\R^d)}
\lesssim \|P_Mg\|_{\dot H^{s_c}(\R^d)}.\label{17.11}
\end{align}
\end{cor}
\begin{proof}
We write
\begin{subequations}\label{est:cor-3.4}
\begin{align}
\Big\|&\langle t^{\alpha}|\nabla|\rangle^{\beta}|\nabla|^{s_c+\gamma}\big(\chi_{\le \frac1{10}M|t|}e^{it\Delta}\big(\chi_{\le 10}P_Mg\big)\big)\Big\|_{L^q_tL^r_x(\R\times\R^d)}\notag\\
&\lesssim
\Big\|\langle t^{\alpha}|\nabla|\rangle^{\beta}|\nabla|^{s_c+\gamma}P_{\le |t|^{-\alpha}}\big(\chi_{\le \frac1{10}M|t|}e^{it\Delta}\big(\chi_{\le 10}P_Mg\big)\big)\Big\|_{L^q_tL^r_x(\R\times\R^d)}\label{cor3.4-1}\\
&\qquad+\Big\|\langle t^{\alpha}|\nabla|\rangle^{\beta}|\nabla|^{s_c+\gamma}P_{\ge |t|^{-\alpha}}\big(\chi_{\le \frac1{10}M|t|}e^{it\Delta}\big(\chi_{\le 10}P_Mg\big)\big)\Big\|_{L^q_tL^r_x(\R\times\R^d)}.\label{cor3.4-2}
\end{align}
\end{subequations}
For the first term \eqref{cor3.4-1}, we may use the Sobolev inequality to replace $(\gamma, r)$ by $(\gamma(q), r(q))$ which is defined by (see Remark \ref{rem:derivatives-radial})
\begin{align}\label{parameter-up}
\gamma(q)=\frac2q\cdot \frac{d-1}{2d-1}-,\quad 
\frac1{r(q)}=\frac12-\frac{2}{q}\cdot \frac{1}{2d-1}-.
\end{align}
Then we set $s_*\ge -\gamma(q)$ such that $s_c+\gamma(q)\ge 0$. To simplify the notation, we still denote $(\gamma(q), r(q))$ by $(\gamma, r)$.  Then 
\begin{align*}
 \eqref{cor3.4-1}
&\lesssim
\Big\||\nabla|^{s_c+\gamma}\big(\chi_{\le \frac1{10}M|t|}e^{it\Delta}\big(\chi_{\le 10}P_Mg\big)\big)\Big\|_{L^q_tL^r_x(\R\times\R^d)}.
\end{align*}
Noting that $\gamma(q)<\frac{d}{r(q)}$, then using Lemma \ref{lem:frac_Hs} and  Lemma \ref{lem:radial-Str},
it is controlled by $\|P_Mg\|_{\dot H^{s_c}(\R^d)}$.


For the second term  \eqref{cor3.4-2}, noting that $\chi_{\le R}=P_{\le \frac{20}R}\>\chi_{\le R}$,   we have that 
\begin{align*}
\eqref{cor3.4-2}
&\lesssim
\Big\||t|^{\alpha\beta}|\nabla|^{s_c+\beta+\gamma}P_{\ge |t|^{-\alpha}}\big(\chi_{\le \frac1{10}M|t|}e^{it\Delta}\big(\chi_{\le 10}P_Mg\big)\big)\Big\|_{L^q_tL^r_x(\R\times\R^d)}\notag\\
&\lesssim
\Big\||t|^{\alpha\beta}|\nabla|^{s_c+\beta+\gamma}P_{\ge |t|^{-\alpha}}P_{\le \frac{200}{M|t|}+2M}\big(\chi_{\le \frac1{10}M|t|}e^{it\Delta}\big(\chi_{\le 10}P_Mg\big)\big)\Big\|_{L^q_tL^r_x(\R\times\R^d)}.
\end{align*}

Since $\alpha\ge 1$, then due to the compatibility of frequency restriction
$$
P_{\ge |t|^{-\alpha}}P_{\le \frac{200}{M|t|}+2M},
$$
we have that 
$
|t|\gtrsim 1.
$
Hence, we get that 
\begin{align}
\eqref{cor3.4-2}\lesssim \Big\||t|^{\alpha\beta}|\nabla|^{s_c+\beta+\gamma}P_{\ge |t|^{-\alpha}}P_{\le \frac{200}{M|t|}+2M}\big(\chi_{\le \frac1{10}M|t|}e^{it\Delta}\big(\chi_{\le 10}P_Mg\big)\big)\Big\|_{L^q_tL^r_x(\{|t|\gtrsim1\}\times\R^d)}. \label{10.25}
\end{align}

For \eqref{10.25}, if $s_c+\beta+\gamma\ge 0$, then applying Lemma \ref{lem:local-fre-I}, we obtain that 
\begin{align}
\eqref{10.25}
\lesssim 
\big\|t^{-d(\frac12-\frac1{ r})+\alpha\beta}\big\|_{L^q_t(\{|t|\gtrsim 1\})}M^{-K}\|P_Mg\|_{\dot H^{s_c}(\R^d)}
\lesssim 
\|P_Mg\|_{\dot H^{s_c}(\R^d)}.
\end{align}
Here we have used the condition of $  d(\frac12-\frac 1r)>\alpha\beta+\frac1q$. 
If $s_c+\beta+\gamma< 0$, then by Bernstein's inequality and Lemma \ref{lem:local-fre-I}, we get that 
\begin{align*}
\eqref{10.25}
\lesssim 
&\Big\| |t|^{-(s_c+\gamma)\alpha}P_{\ge |t|^{-\alpha}}P_{\le \frac{200}{M|t|}+2M}\big(\chi_{\le \frac1{10}M|t|}e^{it\Delta}\big(\chi_{\le 10}P_Mg\big)\big)\Big\|_{L^q_tL^r_x(\{|t|\gtrsim1\}\times\R^d)}\\
\lesssim 
&\Big\| |t|^{-(s_c+\gamma)\alpha}\big(\chi_{\le \frac1{10}M|t|}e^{it\Delta}\big(\chi_{\le 10}P_Mg\big)\big)\Big\|_{L^q_tL^r_x(\{|t|\gtrsim1\}\times\R^d)}\\
\lesssim 
&\big\|t^{-d(\frac12-\frac1{ r})-(s_c+\gamma)\alpha}\big\|_{L^q_t(\{|t|\gtrsim 1\})}M^{-K}\|P_Mg\|_{\dot H^{s_c}(\R^d)}\\
\lesssim 
&\|P_Mg\|_{\dot H^{s_c}(\R^d)}.
\end{align*}
Here we have used the condition of $d(\frac12-\frac 1r)>-\alpha(s_c+\gamma) +\frac1q$.
Hence, no matter in which cases, \eqref{10.25} is controlled by $\|P_Mg\|_{\dot H^{s_c}(\R^d)}$.  Thus we have that 
\begin{align*}
\eqref{cor3.4-2}
\lesssim 
&\|P_Mg\|_{\dot H^{s_c}(\R^d)}.
\end{align*}

Together with the two estimates on \eqref{est:cor-3.4}, we get \eqref{17.11}.
\end{proof}

The second lemma shows the estimates of the linear flow in the domain far away from the origin.
\begin{lem}\label{lem:g-out-rg}
Let $M\ge1$, $2\le r<\infty$ and $s \ge 0$, then for any $t:|t|\gtrsim \frac 1 M$,
\begin{align}
\Big\||\nabla|^s \Big(\chi_{\ge \frac1{10}M|t|}e^{it\Delta}\big(\chi_{\le 10}P_Mg\big)\Big)\Big\|_{L^r_x(\R^d)}
\lesssim 
M^{-(d-2)\big(\frac12-\frac1r\big)+s }|t|^{-(d-1)\big(\frac12-\frac1r\big)}\left\|P_Mg\right\|_{L^2_x(\R^d)}.\label{eq:18.41}
\end{align}
\end{lem}
\begin{proof}
From the radial Sobolev embedding in Lemma \ref{lem:radial-Sob}, we have
\begin{align}
\Big\|\chi_{\ge \frac1{10}M|t|}e^{it\Delta}\big(\chi_{\le 10}P_Mg\big)\Big\|_{L^r_x(\R^d)}
\lesssim &
\big(M|t|\big)^{-(d-1)\big(\frac12-\frac1r\big)}\left\||\nabla|^{\frac12-\frac1r}e^{it\Delta}\big(\chi_{\le 10}P_Mg\big)\right\|_{L^2_x(\R^d)}\notag\\
\lesssim &
\big(M|t|\big)^{-(d-1)\big(\frac12-\frac1r\big)}\left\||\nabla|^{\frac12-\frac1r}\big(\chi_{\le 10}P_Mg\big)\right\|_{L^2_x(\R^d)}.\label{20.31-I}
\end{align}
Using Lemma \ref{lem:frac_Hs}, we have
$$
\left\||\nabla|^{\frac12-\frac1r}\big(\chi_{\le 10}P_Mg\big)\right\|_{L^2_x(\R^d)}
\lesssim
M^{\frac12-\frac1r}\left\|P_Mg\right\|_{L^2_x(\R^d)}.
$$
This last estimate combined with \eqref{20.31-I} yields
\begin{align}
\Big\|\chi_{\ge \frac1{10}M|t|}e^{it\Delta}\big(\chi_{\le 10}P_Mg\big)\Big\|_{L^r_x(\R^d)}
\lesssim &
M^{-(d-2)\big(\frac12-\frac1r\big)}|t|^{-(d-1)\big(\frac12-\frac1r\big)}\left\|P_Mg\right\|_{L^2_x(\R^d)}.\label{eq:18.40}
\end{align}

Similarly, we also obtain that for any $s \ge 0$, we have \eqref{eq:18.41}.
Indeed, if the derivatives hit the cut-off functions $\chi_{\ge \frac{1}{10}M|t|}$ (since $M|t|\gtrsim 1$) and $\chi_{\le 10}$, the analogous estimates become better. Hence
by the same way as \eqref{eq:18.40}, we obtain the estimates above.
\end{proof}
A consequence of Lemma \ref{lem:g-out-rg} is
\begin{cor} \label{cor:g-out-rg}
Let $(q,r,\gamma)$  be the triple satisfying the same conditions as in Proposition \ref{prop:lineares-cpt-largetime-I}, and let  $\alpha, \beta$ be the constants satisfying
$\alpha\ge1, \beta> 0$ and
\begin{align}
(d-1)(\frac12-\frac1r)>\max\{\alpha\beta,-\alpha(s_c+\gamma)\}+\frac1q,
\label{17.34-I}
\end{align}
Then there exists $s_{*,2}=s_{*,2}(q)<0$ such that for any $s_c\in [s_{*,2},0)$, the following estimate holds,
\begin{align}
\Big\|&\langle t^{\alpha}|\nabla|\rangle^{\beta}|\nabla|^{s_c+\gamma}\big(\chi_{\ge \frac1{10}M|t|}e^{it\Delta}(\chi_{\le 10}P_Mg)\big)\Big\|_{L^q_tL^r_x(\R\times\R^d)}
\lesssim \|P_Mg\|_{\dot H^{s_c}(\R^d)}.\label{17.12}
\end{align}
\end{cor}
\begin{proof}
We decompose it into the following three terms.
\begin{subequations}\label{est:cor-3.6}
\begin{align}
\Big\|&\langle t^{\alpha}|\nabla|\rangle^{\beta}|\nabla|^{s_c+\gamma}\big(\chi_{\ge \frac1{10}M|t|}e^{it\Delta}\big(\chi_{\le 10}P_Mg\big)\big)\Big\|_{L^q_tL^r_x(\R\times\R^d)}\notag\\
&\lesssim
\Big\|\langle t^{\alpha}|\nabla|\rangle^{\beta}|\nabla|^{s_c+\gamma}P_{\le |t|^{-\alpha}}\big(\chi_{\ge \frac1{10}M|t|}e^{it\Delta}\big(\chi_{\le 10}P_Mg\big)\big)\Big\|_{L^q_tL^r_x(\R\times\R^d)}\label{17.26-I}\\
&\qquad+\Big\|\langle t^{\alpha}|\nabla|\rangle^{\beta}|\nabla|^{s_c+\gamma}P_{\ge |t|^{-\alpha}}\big(\chi_{\ge \frac1{10}M|t|}e^{it\Delta}\big(\chi_{\le 10}P_Mg\big)\big)\Big\|_{L^q_tL^r_x(\{|t|\le (3M)^{-\frac1\alpha}\}\times\R^d)}\label{17.26-II}\\
&\qquad+\Big\|\langle t^{\alpha}|\nabla|\rangle^{\beta}|\nabla|^{s_c+\gamma}P_{\ge |t|^{-\alpha}}\big(\chi_{\ge \frac1{10}M|t|}e^{it\Delta}\big(\chi_{\le 10}P_Mg\big)\big)\Big\|_{L^q_tL^r_x(\{|t|\ge (3M)^{-\frac1\alpha}\}\times\R^d)}.\label{17.26-III}
\end{align}
\end{subequations}
For the term \eqref{17.26-I}, treating similarly as \eqref{cor3.4-1}, we  use the Sobolev inequality to replace $(\gamma, r)$ by $(\gamma(q), r(q))$ which is defined in \eqref{parameter-up}. Then we set $s_*\ge -\gamma(q)$ such that $s_c+\gamma(q)\ge 0$. Again, to simplify the notation, we still denote $(\gamma(q), r(q))$ by $(\gamma, r)$.  Then  using Lemma \ref{lem:frac_Hs} and    Lemma \ref{lem:radial-Str}, we have
\begin{align*}
\eqref{17.26-I}
& \lesssim
\Big\||\nabla|^{s_c+\gamma}\big(\chi_{\ge \frac1{10}M|t|}e^{it\Delta}\big(\chi_{\le 10}P_Mg\big)\big)\Big\|_{L^q_tL^r_x(\R\times\R^d)}
\lesssim
\|P_Mg\|_{\dot H^{s_c}(\R^d)}.
\end{align*}

For the term \eqref{17.26-II}, since
$$
|t|^{-\alpha}\ge 3M,
$$
we have
\begin{align*}
\eqref{17.26-II}
& \lesssim
\Big\||t|^{\alpha\beta}|\nabla|^{s_c+\beta+\gamma}P_{\ge |t|^{-\alpha}}\big(\chi_{\ge \frac1{10}M|t|}e^{it\Delta}\big(\chi_{\le 10}P_Mg\big)\big)\Big\|_{L^q_tL^r_x(\{|t|\le (3M)^{-\frac1\alpha}\}\times\R^d)}\\
& \lesssim
\Big\||t|^{\alpha\beta}|\nabla|^{s_c+\beta+\gamma}P_{\ge 3M}\big(\chi_{\ge \frac1{10}M|t|}e^{it\Delta}\big(\chi_{\le 10}P_Mg\big)\big)\Big\|_{L^q_tL^r_x(\{|t|\le (3M)^{-\frac1\alpha}\}\times\R^d)}\\
& \lesssim
M^{-\beta}\Big\||\nabla|^{s_c+\beta+\gamma}P_{\ge 3M}\big(\chi_{\ge \frac1{10}M|t|}e^{it\Delta}\big(\chi_{\le 10}P_Mg\big)\big)\Big\|_{L^q_tL^r_x(\R\times\R^d)}.
\end{align*}
Then using  Lemma \ref{lem:mismatch} twice, Lemma \ref{lem:frac_Hs} and Lemma \ref{lem:radial-Str}, it is bounded by
\begin{align*}
M^{-10}\|P_Mg\|_{\dot H^{s_c}(\R^d)}.
\end{align*}
Therefore, we obtain
\begin{align*}
\eqref{17.26-II} \lesssim
M^{-10}\|P_Mg\|_{\dot H^{s_c}(\R^d)}.
\end{align*}
For the term \eqref{17.26-III}, we have
\begin{align}
\eqref{17.26-III}
&\lesssim
\Big\|t^{\alpha\beta}|\nabla|^{s_c+\beta+\gamma}P_{\ge |t|^{-\alpha}}\big(\chi_{\ge \frac1{10}M|t|}e^{it\Delta}\big(\chi_{\le 10}P_Mg\big)\big)\Big\|_{L^q_tL^r_x(\{|t|\ge (3M)^{-\frac1\alpha}\}\times\R^d)}.\label{12.15}
\end{align}
If $s_c+\beta+\gamma\ge 0$, using \eqref{eq:18.41}, \eqref{12.15} is bounded by
\begin{align*}
M^{-(d-2)(\frac12-\frac1r)+\beta+\gamma}\Big\|t^{\alpha\beta-(d-1)(\frac12-\frac1r)}\Big\|_{L^q_t(\{|t|\ge (3M)^{-\frac1\alpha}\})}\|P_Mg\|_{\dot H^{s_c}(\R^d)}.
\end{align*}
Using the condition of $(d-1)(\frac12-\frac1r)>\alpha\beta+\frac1q$, it is dominated by
\begin{align}
M^{-(d-2)(\frac12-\frac1r)+(d-1)(\frac12-\frac1r)\cdot\frac1\alpha-\frac1{q\alpha}+\gamma}\|P_Mg\|_{\dot H^{s_c}(\R^d)}.\label{21.05}
\end{align}
Now we claim that
\begin{align}
-(d-2)(\frac12-\frac1r)+(d-1)(\frac12-\frac1r)\cdot\frac1\alpha-\frac1{q\alpha}+\gamma\le 0.\label{0.50}
\end{align}
Indeed, using \eqref{Str-conditions}, the left-hand side of \eqref{0.50} is equal to
\begin{align*}
\frac{2\alpha-1}{\alpha}\cdot\Big(\frac1q-(d-1)\Big(\frac12-\frac1r\Big)\Big).
\end{align*}
Note that $\alpha\ge 1$, and from \eqref{17.34-I}: $\frac1q<(d-1)\Big(\frac12-\frac1r\Big)$, the last quantity above is negative. Hence, \eqref{0.50} is valid.
Using \eqref{0.50}, we have that \eqref{21.05} and thus \eqref{12.15} is bounded by $\|P_Mg\|_{\dot H^{s_c}(\R^d)}$.

If $s_c+\beta+\gamma< 0$, using the Bernstein inequality, \eqref{12.15} is bounded by
\begin{align*}
\Big\|t^{-\alpha(s_c+\gamma)}&\big(\chi_{\ge \frac1{10}M|t|}e^{it\Delta}\big(\chi_{\le 10}P_Mg\big)\big)\Big\|_{L^q_tL^r_x(\{|t|\ge (3M)^{-\frac1\alpha}\}\times\R^d)}\\
&\lesssim
M^{-(d-2)(\frac12-\frac1r)-s_c}\Big\|t^{-\alpha(s_c+\gamma)-(d-1)(\frac12-\frac1r)}\Big\|_{L^q_t(\{|t|\ge (3M)^{-\frac1\alpha}\})}\|P_Mg\|_{\dot H^{s_c}(\R^d)}.
\end{align*}
Then similarly as above, and using the condition of $(d-1)(\frac12-\frac1r)+\alpha(s_c+\gamma)>\frac1q$, it is also bounded by
\begin{align*}
M^{-(d-2)(\frac12-\frac1r)+(d-1)(\frac12-\frac1r)\cdot\frac1\alpha-\frac1{q\alpha}+\gamma}\|P_Mg\|_{\dot H^{s_c}(\R^d)}.\end{align*}
Hence,    using \eqref{eq:18.41} again, it is bounded by 
 $\|P_Mg\|_{\dot H^{s_c}(\R^d)}$.
Therefore, we obtain that
\begin{align*}
\eqref{17.26-III}\lesssim \|P_Mg\|_{\dot H^{s_c}(\R^d)}.
\end{align*}
Combining the three estimates on \eqref{est:cor-3.6}, we get \eqref{17.12}.
\end{proof}

Together with  Lemmas \ref{lem:local-fre-I} and \ref{lem:g-out-rg}, Corollaries  \ref{cor:local-fre-I} and \ref{cor:g-out-rg}, we are ready to prove Proposition \ref{prop:lineares-cpt-largetime-I}.
\begin{proof}[Proof of Proposition \ref{prop:lineares-cpt-largetime-I}]
Using Littlewood-Paley's decomposition, we have
\begin{align*}
\Big\||\nabla|^{s}&\big(e^{it\Delta}\big(\chi_{\le 10}P_{\ge N}g\big)\big)\Big\|_{L^r_x(\R^d)}
\lesssim
\sum\limits_{M\ge N}\Big\||\nabla|^{s }\big(e^{it\Delta}\big(\chi_{\le 10}P_Mg\big)\big)\Big\|_{L^r_x(\R^d)}\\
\lesssim &
\sum\limits_{M\ge N}\Big\||\nabla|^{s }\big(\chi_{\le \frac1{10}M|t|}e^{it\Delta}\big(\chi_{\le 10}P_Mg\big)\big)\Big\|_{L^r_x(\R^d)}
+
\sum\limits_{M\ge N}\Big\||\nabla|^{s }\big(\chi_{\ge \frac1{10}M|t|}e^{it\Delta}\big(\chi_{\le 10}P_Mg\big)\big)\Big\|_{L^r_x(\R^d)}.
\end{align*}
Using Lemma \ref{lem:local-fre-I}, we get that for any $K\in\Z^+$,
\begin{align*}
\sum\limits_{M\ge N}\Big\||\nabla|^{s }\big(\chi_{\le \frac1{10}M|t|}e^{it\Delta}\big(\chi_{\le 10}P_Mg\big)\big)\Big\|_{L^r_x(\R^d)}
\lesssim
|t|^{-d(\frac12-\frac1r)}N^{-K}\|P_{\ge N}g\|_{\dot H^{s_c}(\R^d)}.
\end{align*}
Using Lemma \ref{lem:g-out-rg},
\begin{align*}
\sum\limits_{M\ge N}\Big\||\nabla|^{s }\big(\chi_{\ge \frac1{10}M|t|}e^{it\Delta}\big(\chi_{\le 10}P_Mg\big)\big)\Big\|_{L^r_x(\R^d)}\lesssim
N^{-(d-2)\big(\frac12-\frac1r\big)+s-s_c }|t|^{-(d-1)\big(\frac12-\frac1r\big)}\left\|P_{\ge N}g\right\|_{\dot H^{s_c}_x(\R^d)}.
\end{align*}
Combining these estimates, we obtain \eqref{13.26}.

Now we prove \eqref{14.49}. Firstly, we give a reduction as following. Fix $q\ge 2$, and let $\varepsilon_0=\frac1{100q}$. Then to prove \eqref{14.49}, we only need to consider the estimates on the triples $(q,r,\gamma)$ when $\gamma\ge \varepsilon_0$. Indeed, if $\gamma\le \varepsilon_0$, then by the Sobolev inequality it follows from the case when $r$ and $\gamma$ satisfying
\begin{align*}
\frac1r=\frac12-\frac2{dq}+\frac{\varepsilon_0}d;\quad \gamma=\varepsilon_0. 
\end{align*}
Using Littlewood-Paley's decomposition and Lemma \ref{lem:LittlePaley-ineq},
\begin{align*}
\Big\|\langle & t^{\alpha}|\nabla|\rangle^{\beta}|\nabla|^{s_c+\gamma}e^{it\Delta}\big(\chi_{\le 10}P_{\ge N}g\big)\Big\|_{L^q_tL^r_x(\R\times\R^d)}\\
&\lesssim \Big\|\langle t^{\alpha}|\nabla|\rangle^{\beta}|\nabla|^{s_c+\gamma}e^{it\Delta}\big(\chi_{\le 10}\sum\limits_{M\ge N}P_Mg\big)\Big\|_{L^q_tL^r_x(\R\times\R^d)}\\
&\lesssim \Big\|\Big(\sum\limits_{M\ge N}\Big(\langle t^{\alpha}|\nabla|\rangle^{\beta}|\nabla|^{s_c+\gamma}e^{it\Delta}\big(\chi_{\le 10}P_Mg\big)\Big)^2\Big)^\frac12\Big\|_{L^q_tL^r_x(\R\times\R^d)}.
\end{align*}
Since $q\ge 2,r\ge 2$, it is dominated by
$$
 \Big(\sum\limits_{M\ge N}\Big\|\langle t^{\alpha}|\nabla|\rangle^{\beta}|\nabla|^{s_c+\gamma}e^{it\Delta}\big(\chi_{\le 10}P_Mg\big)\Big\|_{L^q_tL^r_x(\R\times\R^d)}^2\Big)^\frac12.
$$
Therefore, we obtain
\begin{align*}
\Big\|\langle & t^{\alpha}|\nabla|\rangle^{\beta}|\nabla|^{s_c+\gamma}\big(e^{it\Delta}\big(\chi_{\le 10}P_{\ge N}g\big)\big)\Big\|_{L^q_tL^r_x(\R\times\R^d)}\\
&\lesssim \Big(\sum\limits_{M\ge N}\Big\|\langle t^{\alpha}|\nabla|\rangle^{\beta}|\nabla|^{s_c+\gamma}e^{it\Delta}\big(\chi_{\le 10}P_Mg\big)\Big\|_{L^q_tL^r_x(\R\times\R^d)}^2\Big)^\frac12.
\end{align*}
Now we check the conditions \eqref{17.33-I} and \eqref{17.34-I}. Setting
$$
s_*=\max\{-\varepsilon_0, s_{*,1},s_{*,2}\},
$$
then $s_c+\gamma\ge 0$. Hence, the conditions \eqref{17.33-I} and \eqref{17.34-I} reduce to
$$
(d-1)(\frac12-\frac1r)>\alpha\beta+\frac1q,
$$
which is valid by choosing $\alpha\beta$ small enough. Then by Corollaries \ref{cor:local-fre-I} and \ref{cor:g-out-rg},
we have
\begin{align*}
\Big\|\langle t^{\alpha}|\nabla|\rangle^{\beta} & |\nabla|^{s_c+\gamma}\big(e^{it\Delta}\big(\chi_{\le 10}P_{\ge N}g\big)\big)\Big\|_{L^q_tL^r_x(\R\times\R^d)}\\
&\lesssim \Big(\sum\limits_{M\ge N}\|P_Mg\|_{\dot H^{s_c}(\R^d)}^2\Big)^\frac12
\lesssim \|P_{\ge N}g\|_{\dot H^{s_c}(\R^d)}.
\end{align*}
This proves the proposition.
\end{proof}

\section{Nonlinear flow estimates on localized initial data}

In this section, we give some nonlinear estimates. Firstly, we give some local time and small data estimates.

\subsection{Local theory}

Since $u_0\in \dot H^{s_c}(\R^d)$, we have the following local and small data results, the proofs are standard. However, we give the details for the sake of the completeness. The first is essentially proved by Guo, Wang \cite{GW}.
\begin{lem}\label{lem:local-result-u}
Let $s_0=-\min\{\frac{d-1}{2d-1},\frac{2(d-1)}{(2d-1)(p+1)}\}$, then for any $s_c>s_0$, the following result holds.
For any fixing $\delta_0>0$,  and any radial function $u_0\in \dot H^{s_c}(\R^d)$, there exists $t_0=t_0(u_0,\delta_0)>0$, such that the Cauchy problem \eqref{eqs:NLS-cubic} is well-posed on the time interval $[0,t_0]$. Moreover
the solution $u$ satisfies
\begin{align}
\|u\|_{L^\infty_t\dot H^{s_c}_x([0,t_0]\times\R^d)}\lesssim 1;\quad \big\||\nabla|^{s_c+\gamma} u\big\|_{L^q_tL^r_x([0,t_0]\times\R^d)}\lesssim \delta_0.\label{3.39}
\end{align}
Here the triple $(q,r,\gamma)$ verifies \eqref{Str-conditions} and $\gamma\in [-s_c,-s_0)$.
\end{lem}
\begin{remark}
The result in this lemma improves the index obtained by Guo, Wang \cite{GW}, who proved the local well-posedness in $\dot H^{s_c}(\R^d)$ when $s_c>-\frac{d-1}{2d+1}$ for radial datum. In particular, in this lemma, when $d\ge 4$, the restriction is $ s_c> -\frac{d-1}{2d-1}$ ($s_c>-0.275, -0.388$ when $d=2,3$ respectively).
\end{remark}
\begin{proof}[Proof of Lemma \ref{lem:local-result-u}]
We only show \eqref{3.39} for some $t_0>0$. Then the local well-posedness with the lifespan $[0, t_0)$ is followed by the standard fixed point argument.
In the following, we prove \eqref{3.39} by two cases: $p\le 1$ and $p>1$ separately.

If $p\le 1$,  we denote the parameter $r_1$ as
$$
\frac1{r_1}=\frac12-\frac1d+\frac{\gamma}{d}.
$$
Then for any $s_c>-\frac{d-1}{2d-1}$ and $\gamma\ge -s_c$,
by the Duhamel formula and Lemma \ref{lem:radial-Str}, we have
\begin{align*}
\big\||\nabla|^{s_c+\gamma} u\big\|_{L^2_tL^{r_1}_x([0,t_0]\times\R^d)}
\lesssim \||\nabla|^{s_c+\gamma} e^{it\Delta}u_0\|_{L^2_tL^{r_1}_x([0,t_0]\times\R^d)}+\big\||\nabla|^{s_c+\gamma}(|u|^pu)\big\|_{L^{\tilde q'}_tL^{\tilde r'}_x([0,t_0]\times \R^d)},
\end{align*}
where $(\tilde q,\tilde r)$ satisfies
$$
\tilde q=\frac{2}{1-p},\quad \frac2{\tilde q}+\frac d{\tilde r}=\frac d2-\gamma.
$$
Hence, by Lemma \ref{lem:Frac_Leibniz}, we get
\begin{align*}
\||\nabla|^{s_c+\gamma}u\|_{L^2_tL^{r_1}_x([0,t_0]\times\R^d)}
\lesssim& \||\nabla|^{s_c+\gamma} e^{it\Delta}u_0\|_{L^2_tL^{r_1}_x([0,t_0]\times\R^d)}\\
&\qquad+\big\||\nabla|^{s_c+\gamma} u\big\|_{L^2_tL^{r_1}_x([0,t_0]\times\R^d)}\big\|u\big\|_{L^2_tL^{r_2}_x([0,t_0]\times \R^d)}^p,
\end{align*}
where the parameter $r_2$ satisfies
$$
\frac1{r_2}=\frac12-\frac1d-\frac{s_c}{d}=\frac{2-p}{dp}.
$$
Then by the Sobolev inequality, we obtain that
\begin{align*}
\big\||\nabla|^{s_c+\gamma} u\big\|_{L^2_tL^{r_1}_x([0,t_0]\times\R^d)}
\lesssim \||\nabla|^{s_c+\gamma} e^{it\Delta}u_0\|_{L^2_tL^{r_1}_x([0,t_0]\times\R^d)}+\big\||\nabla|^{s_c+\gamma} u\big\|_{L^2_tL^{r_1}_x([0,t_0]\times\R^d)}^{p+1}.
\end{align*}
Therefore, for any $\delta_0>0$, if
\begin{align}
\big\||\nabla|^{s_c+\gamma} e^{it\Delta}u_0\big\|_{L^2_tL^{r_1}_x([0,t_0]\times\R^d)}\le \delta_0, \label{2.46}
\end{align}
then by the continuity argument,
\begin{align}
\big\||\nabla|^{s_c+\gamma} u\big\|_{L^2_tL^{r_1}_x([0,t_0]\times\R^d)}\lesssim \delta_0.\label{2.46-I}
\end{align}
Note that
\begin{align*}
\big\||\nabla|^{s_c+\gamma} e^{it\Delta}u_0\big\|_{L^2_tL^{r_1}_x(\R\times\R^d)}\lesssim  \|u_0\|_{\dot H^{s_c}(\R^d)},
\end{align*}
 \eqref{2.46} is verified when $t_0=t_0(u_0,\delta_0)$ is small enough, and thus we have \eqref{2.46-I}.

Similarly,
\begin{align*}
\|u\|_{L^\infty_t\dot H^{s_c}_x([0,t_0]\times\R^d)}
\lesssim &\|e^{it\Delta}u_0\|_{L^\infty_t\dot H^{s_c}_x([0,t_0]\times\R^d)}+\big\||\nabla|^{s_c+\gamma}(|u|^pu)\big\|_{L^{\tilde q'}_tL^{\tilde r'}_x([0,t_0]\times \R^d)}\\
\lesssim& \|u_0\|_{\dot H^{s_c}(\R^d)}+\big\||\nabla|^{s_c+\gamma} u\big\|_{L^2_tL^{r_1}_x([0,t_0]\times\R^d)}^{p+1}.
\end{align*}
Then by \eqref{2.46-I}, we obtain that
$$
\|u\|_{L^\infty_t\dot H^{s_c}_x([0,t_0]\times\R^d)}\lesssim 1.
$$
Further, for general triple $(q,r,\gamma)$ verifying \eqref{Str-conditions} and $\gamma\in [-s_c,\frac{d-1}{2d-1})$,
\begin{align*}
\big\||\nabla|^{s_c+\gamma} u\big\|_{L^q_tL^r_x([0,t_0]\times\R^d)}
\lesssim& \||\nabla|^{s_c+\gamma} e^{it\Delta}u_0\|_{L^q_tL^r_x([0,t_0]\times\R^d)}+\big\||\nabla|^{s_c+\gamma}(|u|^pu)\big\|_{L^{\tilde q'}_tL^{\tilde r'}_x([0,t_0]\times \R^d)}\\
\lesssim&\delta_0+\big\||\nabla|^{s_c+\gamma} u\big\|_{L^2_tL^{r_1}_x([0,t_0]\times\R^d)}^{p+1}.
\end{align*}
Hence, we get
$$
\big\||\nabla|^{s_c+\gamma} u\big\|_{L^q_tL^r_x([0,t_0]\times\R^d)}\lesssim \delta_0.
$$

If $p>1$, we denote the parameter $r_3$ as
$$
\frac1{r_3}=\frac12-\frac2{d(p+1)}+\frac{\gamma}{d}.
$$
Then similarly as above, we obtain that for any $s_c>-\frac{2(d-1)}{(2d-1)(p+1)}$ and $\gamma\ge -s_c$,
\begin{align*}
\big\||\nabla|^{s_c+\gamma} u\big\|_{L^{p+1}_tL^{r_3}_x([0,t_0]\times\R^d)}
\lesssim \||\nabla|^{s_c+\gamma} e^{it\Delta}u_0\|_{L^{p+1}_tL^{r_3}_x([0,t_0]\times\R^d)}+\big\||\nabla|^{s_c+\gamma}(|u|^pu)\big\|_{L^1_tL^{r'_4}_x([0,t_0]\times\R^d)},
\end{align*}
where $r_4$ satisfies
$$
\frac d{r_4}=\frac d2-\gamma.
$$
Hence, by Lemma \ref{lem:Frac_Leibniz} and Sobolev's inequality, we get
\begin{align*}
\big\||\nabla|^{s_c+\gamma} u\big\|_{L^{p+1}_tL^{r_3}_x([0,t_0]\times\R^d)}
\lesssim \||\nabla|^{s_c+\gamma} e^{it\Delta}u_0\|_{L^{p+1}_tL^{r_3}_x([0,t_0]\times\R^d)}+\big\||\nabla|^{s_c+\gamma} u\big\|_{L^{p+1}_tL^{r_3}_x([0,t_0]\times\R^d)}^{p+1}.
\end{align*}
Treating similarly as above, by choosing $t_0=t_0(u_0,\delta_0)$ small enough, we obtain that
\begin{align*}
\big\||\nabla|^{s_c+\gamma} u\big\|_{L^{p+1}_tL^{r_3}_x([0,t_0]\times\R^d)}\lesssim \delta_0,
\end{align*}
and thus obtain \eqref{3.39}.
\end{proof}

Now we fix $\delta_0\in (0,1)$ to be an absolute  small constant which will determined later. For simplicity, we set $t_0(u_0, \delta_0)=2$ by rescaling. Moreover,   we set a number $N=N(\delta_0)$ such that
\begin{align}
\|P_{\ge N}u_0\|_{\dot H^{s_c}(\R^d)}\le \delta_0. \label{small-highfreq}
\end{align}

To prove Theorem \ref{thm:main01}, we split the initial data $u_0$ into three parts as
$$
u_0=\chi_{\le 10}\big(P_{\ge N}u_0\big)+P_{\le N}u_0+\chi_{\ge 10}\big(P_{\ge N}u_0\big).
$$
Accordingly, let
$$
v_0=\chi_{\le 10}\big(P_{\ge N}u_0\big),
$$
and $v$ be the solution of the following equation,
 \begin{equation}\label{eqs:NLS-cubic-high}
   \left\{ \aligned
    &i\partial_{t}v+\Delta v= \chi_{\le 1}(t)|v|^p v,
    \\
    &v(0,x)  =v_0.
   \endaligned
  \right.
 \end{equation}
Moreover, let
$$
w_0=\chi_{\ge 10}\big(P_{\ge N}u_0\big)+P_{\le N}u_0,
$$
and $w=u-v$. Then $w$ is the solution of the following equation,
 \begin{equation}\label{eqs:NLS-cubic-w}
   \left\{ \aligned
    &i\partial_{t}w+\Delta w= |u|^pu-\chi_{\le 1}(t)|v|^p v,
    \\
    &w(0,x)  =w_0.
   \endaligned
  \right.
 \end{equation}

Then the second result is a global result with small data in $\dot H^{s_c}$-level.
\begin{lem}\label{lem:smalldata}
For any $s_c>s_0$, the following result holds.
Let $u_0\in \dot H^{s_c}(\R^d)$ be radial, then there exist a small constant $\delta_0$ and a large constant $N$ verifying \eqref{small-highfreq}, such that the Cauchy problem \eqref{eqs:NLS-cubic-high} is globally well-posed. In particular,
the solution $v$ satisfies
$$
\|v\|_{L^\infty_t\dot H^{s_c}_x(\R\times\R^d)}+\big\||\nabla|^{s_c+\gamma} v\big\|_{L^q_tL^r_x(\R\times\R^d)}\lesssim \|v_0\|_{\dot H^{s_c}(\R^d)}.
$$
Here the triple $(q,r,\gamma)$ verifies \eqref{Str-conditions} and $\gamma\in [0,-s_0)$.
\end{lem}
\begin{proof}
We adopt the same notation and argue similarly as in the proof of Lemma \ref{lem:local-result-u}. Moreover, we may assume that $\gamma\ge -s_c$, otherwise it follows by the Sobolev inequality.  In the case of $p\le 1$,  for any $s_c>-\frac{d-1}{2d-1}$, 
\begin{align*}
\big\||\nabla|^{s_c+\gamma} v\big\|_{L^2_tL^{r_1}_x(\R\times\R^d)}
\lesssim \|v_0\|_{\dot H^{s_c}(\R^d)}+\big\||\nabla|^{s_c+\gamma} v\big\|_{L^2_tL^{r_1}_x(\R\times\R^d)}^{p+1}.
\end{align*}
Hence, by the continuity argument and choosing $\delta_0$ small enough, we obtain
\begin{align*}
\big\||\nabla|^{s_c+\gamma} v\big\|_{L^2_tL^{r_1}_x(\R\times\R^d)}
\lesssim \|v_0\|_{\dot H^{s_c}(\R^d)}.
\end{align*}
Using the estimate above, we have the desired results.  In the case of $p> 1$,  for any $s_c>-\frac{2(d-1)}{(2d-1)(p+1)}$,
\begin{align*}
\big\||\nabla|^{s_c+\gamma} v\big\|_{L^{p+1}_tL^{r_3}_x(\R\times\R^d)}
\lesssim  \|v_0\|_{\dot H^{s_c}(\R^d)}+\big\||\nabla|^{s_c+\gamma} v\big\|_{L^{p+1}_tL^{r_3}_x(\R\times\R^d)}^{p+1}.
\end{align*}
Hence, arguing similarly as above, we obtain the desired estimates again.
\end{proof}

\subsection{Nonlinear estimates on $v$}

In this subsection, we give the estimates on $v$. For convenience, we introduce some notation.
We denote
$X(\alpha,\beta)$ be the space with the norm:
$$
\|f\|_{X(\alpha,\beta)}=\big\|\langle t^{\alpha}|\nabla|\rangle^{\beta}|\nabla|^{s_c}P_Mf\big\|_{l^\infty_M L^2_t L^{\frac{2d}{d-2}}_x(\{M\ge 1\}\times\R\times\R^d)}.
$$
Then the main result in this subsection is
\begin{prop}\label{prop:v-X-ab}
Let $v$ be the solution of \eqref{eqs:NLS-cubic-high}, then there exist $\alpha_0\ge 1,\beta_0>0$ and $s_*<0$, such that for any $s_c\in [s_*,0)$,
$$
\|v\|_{X(\alpha_0,\beta_0)}\lesssim \|v_0\|_{\dot H^{s_c}(\R^d)}.
$$
\end{prop}
\begin{proof}
We write
\begin{align}
\big\|\langle &t^{\alpha_0}|\nabla|\rangle^{\beta_0}|\nabla|^{s_c}P_Mv\big\|_{l^\infty_M L^2_t L^{\frac{2d}{d-2}}_x(\{M\ge 1\}\times\R\times\R^d)}\notag\\
&=\big\|\langle t^{\alpha_0}|\nabla|\rangle^{\beta_0}|\nabla|^{s_c}P_Mv\big\|_{l^\infty_M L^2_t L^{\frac{2d}{d-2}}_x(\{M\ge 1\}\times\{|t|\le M^{-\frac1{\alpha_0}}\}\times\R^d)}\label{17.40}\\
&\quad+\big\|\langle t^{\alpha_0}|\nabla|\rangle^{\beta_0}|\nabla|^{s_c}P_Mv\big\|_{l^\infty_M L^2_t L^{\frac{2d}{d-2}}_x(\{M\ge 1\}\times\{|t|\ge M^{-\frac1{\alpha_0}}\}\times\R^d)}.\label{17.41}
\end{align}

\textbf{Estimates on \eqref{17.40}.} Note that
\begin{align*}
\big\|\langle & t^{\alpha_0}|\nabla|\rangle^{\beta_0}|\nabla|^{s_c}P_Mv\big\|_{ L^2_t L^{\frac{2d}{d-2}}_x(\{|t|\le M^{-\frac1{\alpha_0}}\}\times\R^d)}\lesssim
\big\||\nabla|^{s_c}P_Mv\big\|_{L^2_t L^{\frac{2d}{d-2}}_x(\R\times\R^d)}.
\end{align*}
Then by Lemma \ref{lem:smalldata} (where we choose the triple $(q,r,\gamma)=(2,\frac{2d}{d-2},0)$), it is further controlled by
$
\|v_0\|_{\dot H^{s_c}(\R^d)}.
$
Therefore, we have the bound of \eqref{17.40} as
\begin{align}\label{est:17.40}
\big\|\langle t^{\alpha_0}|\nabla|\rangle^{\beta_0}|\nabla|^{s_c}P_Mv\big\|_{l^\infty_M L^2_t L^{\frac{2d}{d-2}}_x(\{M\ge 1\}\times\{|t|\le M^{-\frac1{\alpha_0}}\}\times\R^d)}
\lesssim  \|v_0\|_{\dot H^{s_c}(\R^d)}.
\end{align}

\textbf{Estimates on \eqref{17.41}.} It is controlled by
$$
\big\| t^{\alpha_0\beta_0}|\nabla|^{\beta_0+s_c}P_Mv\big\|_{l^\infty_M L^2_t L^{\frac{2d}{d-2}}_x(\{M\ge 1\}\times\{|t|\ge M^{-\frac1{\alpha_0}}\}\times\R^d)}.
$$
We only consider the positive time, that is, $t\ge 0$,  the negative time being obtained in the same way. Now we write
\begin{align*}
P_Mv=e^{it\Delta}P_M&v_0+\int_0^{\frac12t} e^{i(t-s)\Delta}\chi_{\le 1}(s)P_M(|v|^pv)\,ds\\
&+\int_{\frac12t}^t e^{i(t-s)\Delta}\chi_{\le 1}(s)P_M(|v|^pv)\,ds,
\end{align*}
then we need to consider the following three parts,
\begin{subequations}\label{17.41-123}
\begin{align}
&\big\|t^{\alpha_0\beta_0}|\nabla|^{\beta_0+s_c}e^{it\Delta}P_Mv_0\big\|_{l^\infty_M L^2_t L^{\frac{2d}{d-2}}_x(\{M\ge 1\}\times\{|t|\ge M^{-\frac1{\alpha_0}}\}\times\R^d)};\label{17.41-I}\\
&\Big\|t^{\alpha_0\beta_0}|\nabla|^{\beta_0+s_c}\!\! \int_0^{\frac12t} e^{i(t-s)\Delta}\chi_{\le 1}(s)P_M(|v|^pv)\,ds\Big\|_{l^\infty_M L^2_t L^{\frac{2d}{d-2}}_x(\{M\ge 1\}\times\{|t|\ge M^{-\frac1{\alpha_0}}\}\times\R^d)};\label{17.41-II}
\end{align}
and
\begin{align}
\Big\|t^{\alpha_0\beta_0}|\nabla|^{\beta_0+s_c}\!\! \int_{\frac12t}^t e^{i(t-s)\Delta}\chi_{\le 1}(s)P_M(|v|^pv)\,ds\Big\|_{l^\infty_M L^2_t L^{\frac{2d}{d-2}}_x(\{M\ge 1\}\times\{|t|\ge M^{-\frac1{\alpha_0}}\}\times\R^d)}.\label{17.41-III}
\end{align}
\end{subequations}

\textbf{Estimates on \eqref{17.41-I}.} Here we choose $s_*=s_*(2),\alpha_0\ge 1$ and $\beta_0>0$ with $\alpha\beta\le \eta_*(2,\frac{2d}{d-2},0)$, where $s_*, \eta_*$ are the parameters obtained in Proposition \ref{prop:lineares-cpt-largetime-I} (we may narrow $s_*$ suitably in the following if necessary). Then by  Proposition \ref{prop:lineares-cpt-largetime-I},  we obtain that for any $s_c\ge s_*$,
\begin{align*}
\big\|t^{\alpha_0\beta_0}|\nabla|^{\beta_0+s_c}e^{it\Delta}P_Mv_0\big\|_{l^\infty_M L^2_t L^{\frac{2d}{d-2}}_x(\{M\ge 1\}\times\{|t|\ge M^{-\frac1{\alpha_0}}\}\times\R^d)}
\lesssim \|v_0\|_{\dot H^{s_c}(\R^d)}.
\end{align*}

\textbf{Estimates on \eqref{17.41-II}.}  From Lemma \ref{lem:formula-St},
\begin{align*}
\Big\|t^{\alpha_0\beta_0}&|\nabla|^{\beta_0+s_c}\!\! \int_0^{\frac12t}e^{i(t-s)\Delta}\chi_{\le 1}(s)P_M(|v|^pv)\,ds\Big\|_{L^2_t L^{\frac{2d}{d-2}}_x(\{|t|\ge M^{-\frac1{\alpha_0}}\}\times\R^d)}\\
&\lesssim
\Big\|t^{\alpha_0 \beta_0}\!\! \int_0^{\frac12t}  |t-s|^{-1}\chi_{\le 1}(s)\Big\||\nabla|^{\beta_0+s_c}P_M(|v|^pv)\Big\|_{L^{\frac{2d}{d+2}}(\R^d)}\,ds\Big\|_{L^2_t(\{|t|\ge M^{-\frac1{\alpha_0}}\})}\\
&\lesssim
\Big\|t^{\alpha_0\beta_0-1}\!\! \int_0^2 \Big\||\nabla|^{\beta_0+s_c}P_M(|v|^pv)\Big\|_{L^{\frac{2d}{d+2}}_x(\R^d)}\,ds\Big\|_{L^2_t(\{|t|\ge M^{-\frac1{\alpha_0}}\})},
\end{align*}
where we have used the relationship $|t-s|\sim |t|$. We can
 choose $\alpha_0\beta_0$ small enough, such that
 $
 \alpha_0\beta_0<\frac12.
 $
 Then taking $L^2_t$ first and using Bernstein's inequality,  the inequality above is bounded by
\begin{align}
&\int_0^2 \Big\||\nabla|^{s_c+\frac1{2\alpha_0}}P_M(|v|^pv)\Big\|_{L^{\frac{2d}{d+2}}_x(\R^d)}\,ds.\label{3.51}
\end{align}
Now we consider the following two cases. The first case is $s_c+\frac1{2\alpha_0}\le 0.$ Then \eqref{3.51} is dominated by
$$
\int_0^2 \big\||v|^pv\big\|_{L^{\frac{2d}{d+2}}_x(\R^d)}\,ds.
$$
Using  the H\"older inequality, it is further controlled by
\begin{align}
\int_0^2 \big\|v\big\|_{L^{\frac{2d(p+1)}{d+2}}_x(\R^d)}^{p+1}\,ds.\label{13.48-1}
\end{align}
Let $q_1$ verify
$$
\frac1{q_1}=\frac1p-\frac {d+2}{4(p+1)}.
$$
Then $(q_1,\frac{2d(p+1)}{d+2},-s_c)$ verifies \eqref{Str-conditions} (decreasing the distance between $p_0(d)$ and $\frac4d$ to satisfy the conditions in \eqref{Str-conditions} if necessary).

Note that $q_1\ge p+1$ when $s_*$ is close enough to zero (indeed, if $s_c=0$, then $q_1=2(p+1)$), and thus \eqref{13.48-1} is bounded by
$$
\big\|v\big\|_{L^{q_1}_tL^{\frac{2d(p+1)}{d+2}}_x(\R\times\R^d)}^{p+1}.
$$
Using Lemma \ref{lem:smalldata}, it is bounded again by
$
\|v_0\|_{\dot H^{s_c}(\R^d)}^{p+1}.
$
Hence, we obtain
\begin{align*}
\Big\|t^{\alpha_0\beta_0}&|\nabla|^{\beta_0+s_c}\!\! \int_0^{\frac12t}  e^{i(t-s)\Delta}\chi_{\le 1}(s)P_M(|v|^pv)\,ds\Big\|_{L^2_t L^{\frac{2d}{d-2}}_x(\{|t|\ge M^{-\frac1{\alpha_0}}\}\times\R^d)}
\lesssim\|v_0\|_{\dot H^{s_c}(\R^d)}^{p+1}.
\end{align*}
The second case is $s_c+\frac1{2\alpha_0}> 0.$ Then \eqref{3.51} is bounded by
$$
\int_0^2 \big\||\nabla|^{s_c+\frac1{2\alpha_0}}(|v|^pv)\big\|_{L^{\frac{2d}{d+2}}_x(\R^d)}\,ds.
$$
Then using Lemma \ref{lem:Frac_Leibniz} and the H\"older inequality, it is further controlled by
\begin{align}
\int_0^2 \big\||\nabla|^{s_c+\frac1{2\alpha_0}}v\big\|_{L^{\frac{2d(p+1)}{d+2}}_x(\R^d)}\|v\|_{L^{\frac{2d(p+1)}{d+2}}_x(\R^d)}^p\,ds.\label{13.48}
\end{align}
Let $q_2$ verify
$$
\frac1{q_2}=\frac d4+\frac1{4\alpha_0}-\frac {d+2}{4(p+1)},
$$
then for suitable large $\alpha_0$ and small $|s_*|$, $(q_2,\frac{2d(p+1)}{d+2},-\frac1{2\alpha_0})$ verifies \eqref{Str-conditions}. Moreover, we have
$$
\frac1{q_2}+\frac p{q_1}\le 1.
$$
(In particular, if $s_c=0,\alpha_0=+\infty$, then $q_1=q_2=2(p+1)$, hence the conclusions verify when we choose $|s_*|$ small enough and $\alpha_0$ large enough).
Hence, \eqref{13.48} is bounded by
$$
\big\||\nabla|^{s_c+\frac1{2\alpha_0}}v\big\|_{L^{q_2}_tL^{\frac{2d(p+1)}{d+2}}_x(\R\times\R^d)} \big\|v\big\|_{L^{q_1}_tL^{\frac{2d(p+1)}{d+2}}_x(\R\times\R^d)}^{p}.
$$
Using Lemma \ref{lem:smalldata} again, it is bounded by
$
\|v_0\|_{\dot H^{s_c}(\R^d)}^{p+1}.
$
Hence, we also obtain
\begin{align*}
\Big\|t^{\alpha_0\beta_0}&|\nabla|^{\beta_0+s_c}\!\! \int_0^{\frac12t}  e^{i(t-s)\Delta}\chi_{\le 1}(s)P_M(|v|^pv)\,ds\Big\|_{L^2_t L^{\frac{2d}{d-2}}_x(\{|t|\ge M^{-\frac1{\alpha_0}}\}\times\R^d)}
\lesssim\|v_0\|_{\dot H^{s_c}(\R^d)}^{p+1}.
\end{align*}
Therefore, we get
\begin{align}
\Big\|t^{\alpha_0\beta_0}&|\nabla|^{\beta_0+s_c}\!\! \int_0^{\frac12t}  e^{i(t-s)\Delta}\chi_{\le 1}(s)P_M(|v|^pv)\,ds\Big\|_{l^\infty_M L^2_t L^{\frac{2d}{d-2}}_x(\{M\ge 1\}\times\{|t|\ge M^{-\frac1{\alpha_0}}\}\times\R^d)}\lesssim\|v_0\|_{\dot H^{s_c}(\R^d)}^{p+1}. \label{est:17.41-II}
\end{align}

\textbf{Estimates on \eqref{17.41-III}.} By the Sobolev and the Bernstein inequalities, we have
\begin{align}
\Big\|t^{\alpha_0\beta_0}&|\nabla|^{\beta_0+s_c}\!\! \int_{\frac12t}^t e^{i(t-s)\Delta}\chi_{\le 1}(s)P_M(|v|^pv)\,ds\Big\|_{L^2_t L^{\frac{2d}{d-2}}_x(\{|t|\ge M^{-\frac1{\alpha_0}}\}\times\R^d)}\notag\\
&\lesssim
M^{\beta_0+s_c}\Big\|\int_{\frac12t}^t e^{i(t-s)\Delta}\chi_{\le 1}(s)s^{\alpha_0\beta_0}P_M(|v|^pv)\,ds\Big\|_{L^2_t L^{\frac{2d}{d-2}}_x(\R\times\R^d)}.\label{11.18}
\end{align}
Now we split it into two cases: $p\le 1$ and $p>1$.

If $p\le 1$, using   Lemma \ref{lem:strichartz} and \eqref{11.18}, \eqref{17.41-III} is further bounded by
\begin{align}
M^{\beta_0+s_c}\Big\|\chi_{\le 1}(t)t^{\alpha_0\beta_0}P_M(|v|^pv)\Big\|_{L^{\frac{2}{p+1}}_t L^{r_5'}_x(\R\times\R^d)},\label{13.25}
\end{align}
where $r_5$ is the parameter satisfying
$$
\frac1{r_5}=\frac12-\frac{1-p}d.
$$
In particular, this parameter verifies the following  H\"older inequality,
\begin{align}\label{Ho-r1}
\big\||f|^pg\big\|_{L^{r_5'}}
\lesssim \|f\|_{L^{\frac{dp}{2-p}}_x(\R^d)}^p\big\|g\big\|_{L^{\frac{2d}{d-2}}_x(\R^d)}.
\end{align}
Now we consider the term
$$
\big\|P_M(|v|^pv)\big\|_{L^{r_5'}_x(\R^d)}.
$$
We write
\begin{align}
\big\|P_M(|v|^pv)\big\|_{L^{r_5'}_x(\R^d)}
\le  &
\big\|P_M\big(|P_{\le  M}v|^pP_{\le  M}v\big)\big\|_{L^{r_5'}_x(\R^d)}\label{12.48-I}\\
&+
\big\|P_M\big(|v|^pv-|P_{\le  M}v|^pP_{\le  M}v\big)\big\|_{L^{r_5'}_x(\R^d)}.\label{12.48-II}
\end{align}
We choose $s_*<0$ suitably close to 0 such that for any $s_c\in (s_*,0)$,
$$
s_c+\beta_0> 0.
$$
Then for  \eqref{12.48-I}, by Bernstein's inequality, we have
\begin{align*}
\big\|P_M&\big(|P_{\le  M}v|^pP_{\le  M}v\big)\big\|_{L^{r_5'}_x(\R^d)}\\
&\lesssim
M^{-(s_c+\beta_0+\epsilon)}\big\||\nabla|^{s_c+\beta_0+\epsilon}P_M\big(|P_{\le  M}v|^pP_{\le  M}v\big)\big\|_{L^{r_5'}_x(\R^d)},
\end{align*}
where $\epsilon$ is a small positive constant such that $s_c+\beta_0+\epsilon<p+1$. Then by Lemma \ref{lem:Frac_Leibniz}, we further obtain
\begin{align*}
\big\|P_M&\big(|P_{\le  M}v|^pP_{\le  M}v\big)\big\|_{L^{r_5'}_x(\R^d)}\\
&\lesssim
M^{-(s_c+\beta_0+\epsilon)}\big\||\nabla|^{s_c+\beta_0+\epsilon}P_{\le  M} v\big\|_{L^{\frac{2d}{d-2}}_x(\R^d)}
\big\|P_{\le  M}v\big\|_{L^{\frac{dp}{2-p}}_x(\R^d)}^p.
\end{align*}
Now by Littlewood-Paley's decomposition, we write
\begin{align*}
\big\|&|\nabla|^{s_c+\beta_0+\epsilon}P_{\le  M} v\big\|_{L^{\frac{2d}{d-2}}_x(\R^d)}
\lesssim
\big\||\nabla|^{s_c+\beta_0+\epsilon}P_{\le 1} v\big\|_{L^{\frac{2d}{d-2}}_x(\R^d)}+\sum\limits_{1\le M_1\le  M}M_1^{\epsilon}\big\||\nabla|^{s_c+\beta_0}P_{M_1} v\big\|_{L^{\frac{2d}{d-2}}_x(\R^d)}.
\end{align*}
Note that by Lemma \ref{lem:smalldata},
$$
\big\||\nabla|^{s_c+\beta_0+\epsilon}P_{\le 1} v\big\|_{L^\infty_tL^{\frac{2d}{d-2}}_x(\R\times\R^d)}
\lesssim \big\||\nabla|^{s_c}P_{\le 1} v\big\|_{L^\infty_tL^2_x(\R\times\R^d)}
\lesssim \|v_0\|_{\dot H^{s_c}(\R^d)}.
$$
Hence, we obtain that
\begin{align}
M^{\beta_0+s_c}\Big\|&\chi_{\le 1}(t)t^{\alpha_0\beta_0}P_M\big(|P_{\le  M}v|^pP_{\le  M}v\big)\Big\|_{L^{\frac{2}{p+1}}_t L^{r_5'}_x(\R\times\R^d)}\notag\\
\lesssim&
M^{-\epsilon}\Big\|\chi_{\le 1}(t)t^{\alpha_0\beta_0}\Big(\|v_0\|_{\dot H^{s_c}(\R^d)}+\sum\limits_{1\le M_1\le  M}M_1^{\epsilon}\big\||\nabla|^{s_c+\beta_0}P_{M_1} v\big\|_{L^{\frac{2d}{d-2}}_x(\R^d)}\Big)\notag\\
&\quad \cdot\big\|P_{\le  M}v\big\|_{L^{\frac{dp}{2-p}}_x(\R^d)}^p\Big\|_{L^{\frac{2}{p+1}}_t(\R)}\notag\\
\lesssim&
M^{-\epsilon}\|v_0\|_{\dot H^{s_c}(\R^d)}\|v\|_{L^2_tL^{\frac{dp}{2-p}}_x(\R\times\R^d)}^p\notag\\
&\quad +M^{-\epsilon}\Big(\sum\limits_{1\le M_1\le  M}M_1^{\epsilon}\big\|t^{\alpha_0\beta_0}|\nabla|^{s_c+\beta_0}P_{M_1} v\big\|_{L^2_tL^{\frac{2d}{d-2}}_x(\R\times\R^d)}\Big)\cdot\|v\|_{L^2_tL^{\frac{dp}{2-p}}_x(\R\times\R^d)}^p.\label{11.32}
\end{align}
Now by Lemma \ref{lem:smalldata} and the definition of $X(\alpha_0,\beta_0)$, we have
$$
\|v\|_{L^2_tL^{\frac{dp}{2-p}}_x(\R\times\R^d)}\lesssim \|v_0\|_{\dot H^{s_c}(\R^d)},
$$
and
$$
\big\|t^{\alpha_0\beta_0}|\nabla|^{s_c+\beta_0}P_{M_1} v\big\|_{L^2_tL^{\frac{2d}{d-2}}_x(\R\times\R^d)}\lesssim \|v\|_{X(\alpha_0,\beta_0)}.
$$
Inserting these two estimates into \eqref{11.32}, then  \eqref{11.32} is controlled by
\begin{align*}
M^{-\epsilon}\|v_0\|_{\dot H^{s_c}(\R^d)}^{p+1}+M^{-\epsilon}\sum\limits_{1\le M_1\le  M}M_1^{\epsilon}\|v\|_{X(\alpha_0,\beta_0)}\cdot\|v_0\|_{\dot H^{s_c}(\R^d)}^p.
\end{align*}
Taking summation, we obtain
\begin{align}
M^{\beta_0+s_c}\Big\|\chi_{\le 1}(t)&t^{\alpha_0\beta_0}P_M\big(|P_{\le  M}v|^pP_{\le  M}v\big)\Big\|_{L^{\frac{2}{p+1}}_t L^{r_5'}_x(\R\times\R^d)}\notag\\
&\lesssim
\|v_0\|_{\dot H^{s_c}(\R^d)}^{p+1}+\|v_0\|_{\dot H^{s_c}(\R^d)}^p\|v\|_{X(\alpha_0,\beta_0)}.\label{3.27-I}
\end{align}

For \eqref{12.48-II}, by Bernstein's inequality, we have
\begin{align*}
\big\|P_M\big(|v|^pv&-|P_{\le  M}v|^pP_{\le  M}v\big)\big\|_{L^{r_5'}_x(\R^d)}\\
&\lesssim
\big\|P_{\ge  M}v\big\|_{L^{\frac{2d}{d-2}}_x(\R^d)}\big\|v\big\|_{L^{\frac{dp}{2-p}}_x(\R^d)}^p\\
&\lesssim
\sum\limits_{M_1\ge M}\big\|P_{M_1}v\big\|_{L^{\frac{2d}{d-2}}_x(\R^d)}\big\|v\big\|_{L^{\frac{dp}{2-p}}_x(\R^d)}^p\\
&\lesssim
\sum\limits_{M_1\ge M}M_1^{-(s_c+\beta_0)}\big\||\nabla|^{s_c+\beta_0}P_{M_1}v\big\|_{L^{\frac{2d}{d-2}}_x(\R^d)}\big\|v\big\|_{L^{\frac{dp}{2-p}}_x(\R^d)}^p.
\end{align*}
Hence, we obtain that
\begin{align*}
M^{\beta_0+s_c}\Big\|&\chi_{\le 1}(t)t^{\alpha_0\beta_0}P_M\big(|v|^pv-|P_{\le  M}v|^pP_{\le  M}v\big)\Big\|_{L^{\frac{2}{p+1}}_t L^{r_5'}_x(\R\times\R^d)}\notag\\
\lesssim&
M^{\beta_0+s_c}\sum\limits_{M_1\ge M}M_1^{-(s_c+\beta_0)}\Big\|\chi_{\le 1}(t)t^{\alpha_0\beta_0}\big\||\nabla|^{s_c+\beta_0}P_{M_1}v\big\|_{L^{\frac{2d}{d-2}}_x(\R^d)}\big\|v\big\|_{L^{\frac{dp}{2-p}}_x(\R^d)}^p\Big\|_{L^{\frac{2}{p+1}}_t(\R)}\notag\\
\lesssim&
M^{\beta_0+s_c}\sum\limits_{M_1\ge M}M_1^{-(s_c+\beta_0)}\Big\|t^{\alpha_0\beta_0}|\nabla|^{s_c+\beta_0}P_{M_1}v\Big\|_{L^2_tL^{\frac{2d}{d-2}}_x(\R\times\R^d)}\big\|v\big\|_{L^2_tL^{\frac{dp}{2-p}}_x(\R\times\R^d)}^p.
\end{align*}
Similar as above, it is further bounded by
$$
M^{\beta_0+s_c}\sum\limits_{M_1\ge M}M_1^{-(s_c+\beta_0)}\|v\|_{X(\alpha_0,\beta_0)}\cdot\|v_0\|_{\dot H^{s_c}(\R^d)}^p.
$$
Taking summation, we obtain that
\begin{align}
M^{\beta_0+s_c}\Big\|\chi_{\le 1}(t)&t^{\alpha_0\beta_0}P_M\big(|v|^pv-|P_{\le  M}v|^pP_{\le  M}v\big)\Big\|_{L^{\frac{2}{p+1}}_t L^{r_5'}_x(\R\times\R^d)}\notag\\
&\lesssim
\|v_0\|_{\dot H^{s_c}(\R^d)}^p\|v\|_{X(\alpha_0,\beta_0)}.\label{3.27-II}
\end{align}

Now, together with \eqref{13.25}, \eqref{12.48-I}, \eqref{12.48-II}, \eqref{3.27-I} and \eqref{3.27-II}, we obtain the estimates on \eqref{17.41-III} in the case of $p\le 1$ as
\begin{align}
\Big\|t^{\alpha_0\beta_0}|\nabla|^{\beta_0+s_c}&\!\! \int_{\frac12t}^t e^{i(t-s)\Delta}\chi_{\le 1}(s)P_M(|v|^pv)\,ds\Big\|_{L^2_t L^{\frac{2d}{d-2}}_x(\{|t|\ge M^{-\frac1{\alpha_0}}\}\times\R^d)}\notag\\
\lesssim &\|v_0\|_{\dot H^{s_c}(\R^d)}^{p+1}+\|v_0\|_{\dot H^{s_c}(\R^d)}^p\|v\|_{X(\alpha_0,\beta_0)}.\label{est:17.41-III}
\end{align}

Next, we consider the case when $p>1$ (now $d=2,3$), which can be treated similarly as above. Then using   Lemma \ref{lem:strichartz} and \eqref{11.18}, \eqref{17.41-III} is bounded by
\begin{align}
M^{\beta_0+s_c}\Big\|\chi_{\le 1}(t)t^{\alpha_0\beta_0}P_M(|v|^pv)\Big\|_{L^1_t L^2_x(\R\times\R^d)}.\label{13.25-II}
\end{align}
Arguing similarly as the case of $p\le 1$, and based on the H\"older inequality,
$$
\big\||f|^pf\big\|_{L^1_t L^2_x(\R\times\R^d)}\le \|f\|_{L^2_t L^{\frac{2d}{d-2}}_x(\R\times\R^d)}\|f\|_{L^{2p}_t L^{dp}_x(\R\times\R^d)}^p,
$$
we also obtain \eqref{est:17.41-III} when $p>1$.

Now collecting the three estimates on \eqref{17.41-123}, we get  that
\begin{align*}
\eqref{17.41} \lesssim \|v_0\|_{\dot H^{s_c}(\R^d)}+\|v_0\|_{\dot H^{s_c}(\R^d)}^{p+1}+\|v_0\|_{\dot H^{s_c}(\R^d)}^p\|v\|_{X(\alpha_0,\beta_0)}.
\end{align*}
Combining this estimate with \eqref{est:17.40}, we obtain that
$$
\|v\|_{X(\alpha_0,\beta_0)}\lesssim \|v_0\|_{\dot H^{s_c}(\R^d)}+\|v_0\|_{\dot H^{s_c}(\R^d)}^{p+1}+\|v_0\|_{\dot H^{s_c}(\R^d)}^p\|v\|_{X(\alpha_0,\beta_0)}.
$$
Using \eqref{small-highfreq} and choosing $\delta_0$ suitably small, we give the proof of the proposition.
\end{proof}

As a consequence, we have
\begin{cor}\label{cor:smooth-t+1}
There exists $s_*<0$, such that for any $s_c\in [s_*,0)$, the following result holds.
\begin{align*}
\|\chi_{\ge 1}(t)v\|_{L^2_tL^\frac{2d}{d-2}_x([0,2]\times\R^d)}
\lesssim \|v_0\|_{\dot H^{s_c}(\R^d)}.
\end{align*}
\end{cor}
\begin{proof} 
For the low frequency part, by Lemma \ref{lem:smalldata} we have
\begin{align*}
\|P_{\le 1}v\|_{L^2_tL^\frac{2d}{d-2}_x([0,2]\times\R^d)}
\lesssim \|P_{\le 1}|\nabla|^{s_c}v\|_{L^2_tL^\frac{2d}{d-2}_x([0,2]\times\R^d)}
\lesssim \|v_0\|_{\dot H^{s_c}(\R^d)}.
\end{align*}
For the high frequency part, from Proposition \ref{prop:v-X-ab} and Sobolev's inequality, we have that for any $M\ge 1$,
\begin{align*}
\big\|t^{\alpha_1\beta_0}|\nabla|^{\beta_0+s_c}P_Mv\big\|_{L^2_tL^\frac{2d}{d-2}_x([\frac12,2]\times\R^d)}
\lesssim \|v_0\|_{\dot H^{s_c}(\R^d)}.
\end{align*}
This implies that
\begin{align*}
\big\|P_Mv\big\|_{L^2_tL^\frac{2d}{d-2}_x([\frac12,2]\times\R^d)}
\lesssim M^{-(\beta_0+s_c)}\|v_0\|_{\dot H^{s_c}(\R^d)}.
\end{align*}
Choosing $s_*$ small enough such that  $\beta_0+s_c>0$, then 
taking summation on $M$, we have that 
\begin{align*}
\|P_{\ge 1}v\|_{L^2_tL^\frac{2d}{d-2}_x([\frac12,2]\times\R^d)}
\lesssim \|v_0\|_{\dot H^{s_c}(\R^d)}.
\end{align*}
Hence, we obtain the desired estimates.
\end{proof}

Furthermore, we also need  the following long time estimate on $v$.
\begin{prop}\label{lem:est-v-long} 
There exists $s_*<0$, such that the following properties hold.  For any $t\ge 2, s_c \in [s_*, 0]$ and any $r$ such that 
\begin{align}\label{cond-r}
\frac{1-p}2+\frac{d-1}{p+1}|s_*|<\frac1r\le\frac{1-p}2+\frac{p+1}{d}- \frac{(p+1)|s_*|}{d},
\end{align}
then 
\begin{align} \label{est-v-long-1}
\|v(t)\|_{L^r_x(\R^d)}\lesssim t^{-(d-1)(\frac12-\frac 1r)}.
\end{align}
Moreover, 
\begin{align} \label{est-v-long-2}
\|v\|_{L^1_tL^{\frac{2}{1-p}}_x([2,\infty)\times \R^d)}\lesssim 1.
\end{align}
\end{prop}
\begin{proof} First, we consider \eqref{est-v-long-1}. By the Duhamel formula, we have
\begin{align*}
\|v\|_{L^r_x(\R^d)}
\lesssim & \big\|e^{it\Delta}v_0\big\|_{L^r_x(\R^d)}
+\int_0^t\big\|e^{i(t-s)\Delta}\chi_{\le1}(s)|u|^pu\big\|_{L^r_x(\R^d)}\,ds.
\end{align*}
From Proposition \ref{prop:lineares-cpt-largetime-I} and \eqref{small-highfreq}, we get that for $r>2$, 
\begin{align}
\big\|e^{it\Delta}v_0\big\|_{L^r_x(\R^d)}
\lesssim N^{-(d-2)(\frac12-\frac 1r)-s_c}|t|^{-(d-1)(\frac12-\frac 1r)}\|P_{\ge N}u_0\|_{\dot H^{s_c}(\R^d)}
\lesssim  \delta_0|t|^{-(d-1)(\frac12-\frac 1r)}.\label{2.53}
\end{align}
Here we shall choose $|s^*|$ suitably small such that $-(d-2)(\frac12-\frac 1r)+|s_*|\le 0$.

Note that when $t\ge2, s\le \frac{11}{10}$, we have that $|t-s|\sim t$. Then
from Lemma \ref{lem:formula-St}, we have
\begin{align*}
\int_0^t\big\|e^{i(t-s)\Delta}&\chi_{\le1}(s)|u|^pu\big\|_{L^r_x(\R^d)}\,ds\\
\lesssim &
\int_0^t|t-s|^{-d(\frac12-\frac 1r)}\chi_{\le1}(s)\big\||u|^pu\big\|_{L^{r'}_x(\R^d)}\,ds\\
\lesssim &
|t|^{-d(\frac12-\frac 1r)}\int_0^2\big\|u(s)\big\|_{L^{r'(p+1)}_x(\R^d)}^{p+1}\,ds.
\end{align*}
Setting $q$ such that 
$$
q\ge 2,\qquad \frac2q+\frac{d}{r'(p+1)}=\frac d2-s_c.
$$
Note that  the condition of \eqref{cond-r} in the right-hand side assures that $q\ge 2$.  Moreover, when $r$ satisfies that
$
\frac1r>\frac{1-p}2+\frac{d-1}{p+1}|s_*|,
$
then $(q,r'(p+1),-s_c)$ verifies \eqref{Str-conditions} and $q\ge p+1$ (see Remark \ref{rem:derivatives-radial}). 
Hence, by H\"older's inequality and Lemma \ref{lem:local-result-u}, we have
\begin{align*}
\int_0^2\big\|u\big\|_{L^{r'(p+1)}_x(\R^d)}^{p+1}\,ds
\lesssim
\big\|u\big\|_{L^q_tL^{r'(p+1)}_x([0,2]\times\R^d)}^{p+1}
\lesssim 1.
\end{align*}
Hence, we obtain that
\begin{align*}
\int_0^t\big\|e^{i(t-s)\Delta}&\chi_{\le1}(s)|u|^pu\big\|_{L^r_x(\R^d)}\,ds
\lesssim  |t|^{-d(\frac12-\frac 1r)}.
\end{align*}
This last estimate combined with \eqref{2.53}, gives \eqref{est-v-long-1}.

Now we consider \eqref{est-v-long-2}. By the Duhamel formula, we have
\begin{subequations}\label{est-v-long-123}
\begin{align}
\|v\|_{L^1_tL^{\frac{2}{1-p}}_x([2,\infty)\times \R^d)}
\lesssim & \big\|e^{it\Delta}v_0\big\|_{L^1_tL^{\frac{2}{1-p}}_x([2,\infty)\times \R^d)}\label{v-long-1}\\
&+\Big\|\int_0^te^{i(t-s)\Delta}\chi_{\le1}(s)P_{\le 1}(|u|^pu)\,ds\Big\|_{L^1_tL^{\frac{2}{1-p}}_x([2,\infty)\times \R^d)}\label{v-long-2}\\
&+\Big\|\int_0^te^{i(t-s)\Delta}\chi_{\le1}(s)P_{\ge 1}(|u|^pu)\,ds\Big\|_{L^1_tL^{\frac{2}{1-p}}_x([2,\infty)\times \R^d)}.\label{v-long-3}
\end{align}
\end{subequations}
For the term \eqref{v-long-1}, it follows by \eqref{2.53} that 
$$
\big\|e^{it\Delta}v_0\big\|_{L^1_tL^{\frac{2}{1-p}}_x([2,\infty)\times \R^d)}\lesssim 
 \delta_0\big\| |t|^{-(d-1)\frac p2}\big\|_{L^1_t([2,\infty))}
 \lesssim 1. 
$$
Here we shall choose $|s^*|$ suitably small such that $(d-1)p>2$. 
For the  term  \eqref{v-long-2}, by Bernstein's inequality, we have that 
 \begin{align*}
\Big\|\int_0^t  e^{i(t-s)\Delta}&\chi_{\le1}(s)P_{\le 1}(|u|^pu)\,ds\Big\|_{L^1_tL^{\frac{2}{1-p}}_x([2,\infty)\times \R^d)}\\
\lesssim &
\Big\|\int_0^te^{i(t-s)\Delta}\chi_{\le1}(s)|u|^pu\,ds\Big\|_{L^1_tL^{(\frac{r_6}{1+p})'}_x([2,\infty)\times \R^d)}\\
\lesssim &
\Big\|\int_0^2|t-s|^{-d(\frac p2-a_0(1+p))}\big\||u|^pu\big\|_{L^{\frac{r_6}{1+p}}_x(\R^d)}\,ds\Big\|_{L^1_t([2,\infty)}\\
\lesssim &
\big\||t|^{-d(\frac p2-a_0(1+p))}\big\|_{L^1_t([2,\infty)}\big\||u|^pu\big\|_{L^1_tL^{\frac{r_6}{1+p}}_x([0,2]\times\R^d)}
\lesssim 
\big\|u\big\|_{L^{q_6}_tL^{r_6}_x([0,2]\times \R^d)}^{1+p},
\end{align*}
where $a_0, r_6,q_6$ verisfy
$$
a_0=\frac{|s_*|}{d-2};\quad \frac1{r_6}=\frac12-a_0;\quad
\frac1{q_6}=\frac12(a_0d-s_c).
$$
Note that the triple $(q_6,r_6,-s_c)$ verifies \eqref{Str-conditions}, then by Lemma  \ref{lem:local-result-u}, we have that 
 \begin{align*}
\Big\|\int_0^t  e^{i(t-s)\Delta}&\chi_{\le1}(s)P_{\le 1}(|u|^pu)\,ds\Big\|_{L^1_tL^{\frac{2}{1-p}}_x([2,\infty)\times \R^d)}
\lesssim 1.
\end{align*}

For the  term  \eqref{v-long-3}, applying \eqref{est-trun-inho-Stri} (noting that in our case,  $\theta=\frac{d}{3d-2}-$ when $s_c=0$, thus \eqref{theta-cond} verifies when $|s_*|$ is suitable small) to give that
 \begin{align*}
\Big\|\int_0^te^{i(t-s)\Delta}\chi_{\le1}(s)P_{\ge 1}(|u|^pu)\,ds\Big\|_{L^1_tL^{\frac{2}{1-p}}_x([2,\infty)\times \R^d)}
\lesssim &
\big\||\nabla|^{\tilde\gamma}P_{\ge 1}(|u|^pu)\big\|_{L^{\tilde q}_tL^{\tilde r}_x([0,2]\times \R^d)}.
\end{align*}
Here we choose the parameters $\tilde\gamma, \tilde q, \tilde r$ such that 
$$
\frac1{\tilde r}=\frac{p}{r_6}+\frac1{r_0}; \quad 
\frac1{\tilde q}=\frac{p}{q_6}+\frac12; \quad
\tilde\gamma= \Big(\frac12-\frac1{r_0}+a_0p\Big)\Big(\frac1{r_0}-\frac12+\frac{p}{2}\Big)^{-1}\Big(\frac{d-2}{2}-\gamma_0-\frac{(1-p)d}2\Big).
$$
Note that when $|s_*|$ is suitably small, then $\tilde \gamma\le \gamma_0+s_c$ (indeed, when $s_c=0$ then $\tilde \gamma=\frac{d^2}{(2d-1)(3d-2)}+<\gamma_0$). Hence, by Bernstein's and H\"older's inequalities,  and Lemma \ref{lem:Frac_Leibniz}, 
 \begin{align*}
\big\||\nabla|^{\tilde\gamma}P_{\ge 1}(|u|^pu)\big\|_{L^{\tilde q}_tL^{\tilde r}_x([0,2]\times \R^d)}
\lesssim &
\big\||\nabla|^{\gamma_0+s_c}(|u|^pu)\big\|_{L^{\tilde q}_tL^{\tilde r}_x([0,2]\times \R^d)}\\
\lesssim &
\big\||\nabla|^{\gamma_0+s_c}u\big\|_{L^2_tL^{r_0}_x([0,2]\times \R^d)}
\big\|u\big\|_{L^{q_6}_tL^{r_6}_x([0,2]\times \R^d)}^p.
\end{align*}
Therefore, by Lemma \ref{lem:local-result-u}, we get that 
\begin{align*}
\big\||\nabla|^{\tilde\gamma}P_{\ge 1}(|u|^pu)\big\|_{L^{\tilde q}_tL^{\tilde r}_x([0,2]\times \R^d)}\lesssim 1.
\end{align*}
This implies that 
 \begin{align*}
\Big\|\int_0^te^{i(t-s)\Delta}\chi_{\le1}(s)P_{\ge 1}(|u|^pu)\,ds\Big\|_{L^1_tL^{\frac{2}{1-p}}_x([2,\infty)\times \R^d)}
\lesssim &
1.
\end{align*}

Collecting the three estimates on \eqref{est-v-long-123}, we obtain \eqref{est-v-long-2}.
\end{proof}

\section{The proof of Theorem \ref{thm:main01}} \label{Sec:thm01}

In this section, we prove Theorem \ref{thm:main01}.

\subsection{Nonlinear estimates on $w$}
In this subsection, we give some nonlinear estimates of the solution with the low frequency initial data.

First, for $u_0\in \dot H^{s_c}(\R^d),s_c<0$ with supp $u_0\subset \{x:|x|\le 1\}$, we claim that
\begin{align}
w_0\in L^2(\R^d)\quad {and }\quad \|w_0\|_{L^2(\R^d)}\lesssim N^{-s_c}\big\|u_0\big\|_{\dot H^{s_c}(\R^d)}.\label{eq:21.43}
\end{align}
Indeed, by the mismatch estimate in Lemma \ref{lem:mismatch},
$$
\big\|\chi_{\ge 10}\big(P_{\ge N}u_0\big)\big\|_{L^2(\R^d)}\lesssim
N^{-10}\big\|u_0\big\|_{\dot H^{s_c}(\R^d)},
$$
and by the Bernstein estimate,
$$
\big\|P_{\le N}u_0\big\|_{L^2(\R^d)}\lesssim
N^{-s_c}\big\|u_0\big\|_{\dot H^{s_c}(\R^d)}.
$$
This gives \eqref{eq:21.43}.

 To clear our argument, in the following we only consider the case of $p<1$, which is assured by $d\ge 4$.

The first we need is the following local estimates of $w$ in more regular spaces.
\begin{lem}\label{lem:w-local}
The Cauchy problem \eqref{eqs:NLS-cubic-w} is locally well-posed in $L^2(\R^d)$ in the time interval $[0,2]$. In particular,
the solution $w$ satisfies that
$$
\big\|w\big\|_{L^2_tL^{\frac{2d}{d-2}}_x([0,2]\times\R^d)}+\big\|w\big\|_{L^\infty_tL^2_x([0,2]\times\R^d)}\lesssim 1+\|w_0\|_{L^2(\R^d)}.
$$
\end{lem}
\begin{proof}
By   Lemma \ref{lem:strichartz}, we have
$$
\big\|w\big\|_{L^2_tL^{\frac{2d}{d-2}}_x([0,2]\times\R^d)}+\big\|w\big\|_{L^\infty_tL^2_x([0,2]\times\R^d)}
\lesssim \|w_0\|_{L^2(\R^d)}+\big\| |u|^pu- \chi_{\le 1}(t)|v|^pv\big\|_{L^{\frac{2}{p+1}}_t L^{r_5'}_x([0,2]\times\R^d)}.
$$
Next, we consider
$$
\big\| |u|^pu- \chi_{\le 1}(t)|v|^pv\big\|_{L^{\frac{2}{p+1}}_t L^{r_5'}_x([0,2]\times\R^d)}.
$$
Note that
\begin{align*}
\big||u|^pu-\chi_{\le 1}(t)|v|^pv\big|
\lesssim \big(|u|^p+|\tilde{\chi}_{\le 1}(t)v|^p\big)\big(|w|+|\tilde{\chi}_{\ge 1}(t)v|\big).
\end{align*}
Here we denote the time-dependent functions $\tilde{\chi}_{\le 1}(t)=\chi_{\le 1}^\frac1{p+1}(t)$ and $\tilde{\chi}_{\ge 1}(t)=1-\tilde{\chi}_{\le 1}(t)$. 
Hence, by \eqref{Ho-r1}, we have that 
\begin{align}
\big\||u|^pu-&\chi_{\le 1}(t)|v|^pv\big\|_{L^{\frac{2}{p+1}}_t L^{r_5'}_x([0,2]\times\R^d)}\notag\\
\lesssim&
\Big(\|u\|_{L^2_tL^{\frac{dp}{2-p}}_x([0,2]\times\R^d)}^p+\|v\|_{L^2_tL^{\frac{dp}{2-p}}_x([0,2]\times\R^d)}^p\Big)\notag\\
&\quad\cdot
\Big(\|w\|_{L^2_tL^{\frac{2d}{d-2}}_x([0,2]\times\R^d)}+\|\chi_{\ge 1}(t)v\|_{L^2_tL^{\frac{2d}{d-2}}_x([0,2]\times\R^d)}\Big).\label{0.41}
\end{align}
From Lemmas \ref{lem:local-result-u} and \ref{lem:smalldata}, we have that
$$
\|u\|_{L^2_tL^{\frac{dp}{2-p}}_x([0,2]\times\R^d)}\lesssim \delta_0;\quad 
\|v\|_{L^2_tL^{\frac{dp}{2-p}}_x([0,2]\times\R^d)}\lesssim \delta_0.
$$
Moreover, from Corollary \ref{cor:smooth-t+1}, we have that 
$$
\|\chi_{\ge 1}(t)v\|_{L^2_tL^{\frac{2d}{d-2}}_x([0,2]\times\R^d)}
\lesssim \delta_0.
$$
Hence, combining these last two estimates above and \eqref{0.41}, we obtain
\begin{align*}
\big\| |u|^pu-\chi_{\le 1}(t)|v|^pv\big\|_{L^{\frac{2}{p+1}}_t L^{r_5'}_x([0,2]\times\R^d)}
\lesssim
\delta_0^p\big(\|w\|_{L^2_tL^{\frac{2d}{d-2}}_x([0,2]\times\R^d)}+\delta_0\big).
\end{align*}
Therefore, we obtain that 
\begin{align*}
\big\|w\big\|_{L^2_tL^{\frac{2d}{d-2}}_x([0,2]\times\R^d)}+\big\|w\big\|_{L^\infty_tL^2_x([0,2]\times\R^d)}
\lesssim &1+ \|w_0\|_{L^2(\R^d)}+\delta_0^p\|w\|_{L^2_tL^{\frac{2d}{d-2}}_x([0,2]\times\R^d)}. 
\end{align*}
Choosing $\delta_0$ suitably small, we obtain
$$
\big\|w\big\|_{L^2_tL^{\frac{2d}{d-2}}_x([0,2]\times\R^d)}+\big\|w\big\|_{L^\infty_tL^2_x([0,2]\times\R^d)}\lesssim 1+\|w_0\|_{L^2(\R^d)}.
$$
This finishes the proof of the lemma.
\end{proof}

Next, we give the global estimates of $w$.
The following is a modified mass estimate for $\dot H^{s_c}(\R^d)$-datum.
\begin{prop}\label{prop:L2-w}
Let $u_0\in \dot H^{s_c}(\R^d)$ and $I$ be the lifespan of the solution $u$, then there exists $s_*<0$, such that for any $s_c\in (s_*,0)$, the following estimate holds,
$$
\|w\|_{L^\infty_tL^2_x(I\times\R^d)}^2\lesssim N^{-2s_c}\big\|u_0\big\|_{\dot H^{s_c}(\R^d)}^2.
$$
\end{prop}
\begin{proof}
For simplicity, we denote $I=[0,T)$ and  any $t\in I$, 
$$
F(v,w)=|u|^pu-\chi_{\le 1}(t)|v|^p v.
$$
Then from the equation \eqref{eqs:NLS-cubic-w}, we have
\begin{align*}
\partial_t\|w\|_{L^2_x}^2=2\mbox{Im}\int  F(v,w)  \bar w\,dx.
\end{align*}
We may assume $t\ge2$, otherwise, the estimate has been included in Lemma \ref{lem:w-local}. 
Then for any $t\in I, t\ge 2$, integrating in time from $2$ to $t$, we obtain that
\begin{align}
\|w(t)\|_{L^2_x}^2= \| w(2)\|_{L^2}^2+2\mbox{Im}\int_2^t\int_{\R^d}  F(v,w)  \bar w\,dx ds. \label{eq:0.32}
\end{align}

Now we consider
$$
2\mbox{Im}\int_2^t\int_{\R^d}  F(v,w)  \bar w\,dx ds.
$$
To do this, we write
$$
F(v,w)=|u|^pw+|u|^pv-\chi_{\le 1}(t)|v|^p v.
$$
Note that
$$
2\mbox{Im}\int_2^t\int_{\R^d} |u|^pw \bar w\,dx ds=0.
$$ 
Hence, due to the time support, we have that
\begin{align*}
2\mbox{Im}\int_2^t&\int_{\R^d}  F(v,w)  \bar w\,dx ds=2\mbox{Im}\int_2^t\int_{\R^d}  |u|^p v\>\bar w\,dx ds.
\end{align*}
Since  $|u|^p\lesssim |w|^p+|v|^p$, this yields that 
\begin{align}
\Big|2\mbox{Im}\int_2^t\int_{\R^d}  F(v,w)  \bar w\,dx ds\Big|
\lesssim  \int_2^T\int_{\R^d}  |w|^{p+1}|v|\,dx ds
+ \int_2^T\int_{\R^d}  |w||v|^{p+1}\,dx ds.\label{3.00}
\end{align}
For the first term in \eqref{3.00}, we have that 
\begin{align*}
 \int_2^T\int_{\R^d}  |w|^{p+1}|v|\,dx ds
 \lesssim  \|w\|_{L^\infty_t L^2_x([2,T)\times \R^d)}^{p+1}\|v\|_{L^1_t L^{\frac2{1-p}}_x([2,T)\times \R^d)}.
\end{align*}
Thus by Lemmas \ref{lem:est-v-long}, we further get 
\begin{align}\label{est:y1-long}
 \int_2^T\int_{\R^d}  |w|^{p+1}|v|\,dx ds
 \lesssim \|w\|_{L^\infty_tL^2_x([2,T)\times\R^d)}^{p+1}.
\end{align}
For the second term in \eqref{3.00}, we have
\begin{align*}
\int_2^T\int_{\R^d}  |w||v|^{p+1}\,dx ds
\lesssim & \int_2^T\|v(s)\|_{L^{2(p+1)}_x(\R^d)}^{p+1} \|w(s)\|_{L^2_x(\R^d)}\,ds.
\end{align*}
Note that $2(p+1)$ satisfies \eqref{cond-r} and $(d-1)p>2$ when $|s_*|$ suitably small,  we further get 
\begin{align}
\int_2^T\int_{\R^d}  |w||v|^{p+1}\,dx ds
\lesssim & \int_2^T s^{\frac{-(d-1)p}{2}} \|w(s)\|_{L^2_x(\R^d)}\,ds\notag\\
\lesssim & \|w\|_{L^\infty_tL^2_x([0,T)\times\R^d)}. \label{est:y2-long}
\end{align}
Inserting \eqref{est:y1-long} and \eqref{est:y2-long} into \eqref{3.00}, and then \eqref{eq:0.32}, we get that for any $t\in [2,T)$, 
\begin{align*}
\|w(t)\|_{L^2_x}^2\lesssim \| w(2)\|_{L^2}^2+\|w\|_{L^\infty_tL^2_x([0,T)\times\R^d)}+\|w\|_{L^\infty_tL^2_x([0,T)\times\R^d)}^{p+1}.
\end{align*}
This together with the result in Lemma \ref{lem:w-local} implies that 
\begin{align*}
\|w\|_{L^\infty_tL^2_x([0,T)\times\R^d)}^2\lesssim1+ \| w_0\|_{L^2}^2+\|w\|_{L^\infty_tL^2_x([0,T)\times\R^d)}+\|w\|_{L^\infty_tL^2_x([0,T)\times\R^d)}^{p+1}. 
\end{align*}
Since $p<1$, by Cauchy-Schwartz's inequality, we obtain that 
\begin{align*}
\|w\|_{L^\infty_tL^2_x([0,T)\times\R^d)}^2\lesssim 1+ \| w_0\|_{L^2}^2. 
\end{align*}
This combining with \eqref{eq:21.43} finishes the proof of the proposition.
\end{proof}

\subsection{Global existence}
Now we prove $I=\R$. We only consider the positive time, the negative time being obtained in the same way.  By the global result of $v$ obtained in Lemma \ref{lem:smalldata}, we only need to consider the global existence of $w$. 
It follows from the standard bootstrap argument and we only give its sketch. Fixing any $2\le t_0\in I$ and $0<\delta<1$, and treating similarly as in the proof of Lemma \ref{lem:w-local}, we have that 
\begin{align*}
\big\|w\big\|_{L^2_tL^{\frac{2d}{d-2}}_x([t_0,t_0+\delta]\times\R^d)}&+\big\|w\big\|_{L^\infty_tL^2_x([t_0,t_0+\delta]\times\R^d)}\\
\lesssim & \|w(t_0)\|_{L^2(\R^d)}+\Big(\|w\|_{L^2_tL^{\frac{dp}{2-p}}_x([t_0,t_0+\delta]\times\R^d)}^p
+\|v\|_{L^2_tL^{\frac{dp}{2-p}}_x([t_0,t_0+\delta]\times\R^d)}^p\Big)\notag\\
&\quad \cdot\Big(\|w\|_{L^2_tL^{\frac{2d}{d-2}}_x([t_0,t_0+\delta]\times\R^d)}+\|v\|_{L^2_tL^{\frac{2d}{d-2}}_x([t_0,t_0+\delta]\times\R^d)}\Big).
\end{align*}
By the interpolation and H\"older inequality in time, we have that 
\begin{align}\label{interp-sc}
\|w\|_{L^2_tL^{\frac{dp}{2-p}}_x([t_0,t_0+\delta]\times\R^d)}\lesssim \delta^{-\frac{s_c}2} \big\|w\big\|_{L^2_tL^{\frac{2d}{d-2}}_x([t_0,t_0+\delta]\times\R^d)}^{1+s_c}
\big\|w\big\|_{L^\infty_tL^2_x([t_0,t_0+\delta]\times\R^d)}^{-s_c}.
\end{align}
Moreover, by Proposition \ref{lem:est-v-long} and H\"older inequality in time, we have that 
$$
\big\|v\big\|_{L^2_tL^{\frac{dp}{2-p}}_x([t_0,t_0+\delta]\times\R^d)}
+
\|v\|_{L^2_tL^{\frac{2d}{d-2}}_x([t_0,t_0+\delta]\times\R^d)}
\lesssim \delta^\frac12.
$$
Therefore, by Cauchy-Schwartz's inequality,  we obtain that 
\begin{align*}
\big\|w\big\|_{L^2_tL^{\frac{2d}{d-2}}_x([t_0,t_0+\delta]\times\R^d)}&+\big\|w\big\|_{L^\infty_tL^2_x([t_0,t_0+\delta]\times\R^d)}\\
\lesssim &1+ \|w(t_0)\|_{L^2(\R^d)}+\delta^{-\frac{s_c}2} \Big(\big\|w\big\|_{L^\infty_tL^2_x([t_0,t_0+\delta]\times\R^d)}\\
&\quad +\|w\|_{L^2_tL^{\frac{2d}{d-2}}_x([t_0,t_0+\delta]\times\R^d)}+\|w\|_{L^2_tL^{\frac{2d}{d-2}}_x([t_0,t_0+\delta]\times\R^d)}^{p+1}\Big). 
\end{align*}
Choosing $\delta=\delta(\|w(t_0)\|_{L^2(\R^d)})>0$ suitably small and the bootstrap, we obtain
\begin{align}\label{est:w-L2-t0}
\big\|w\big\|_{L^2_tL^{\frac{2d}{d-2}}_x([t_0,t_0+\delta]\times\R^d)}+\big\|w\big\|_{L^\infty_tL^2_x([t_0,t_0+\delta]\times\R^d)}\lesssim 1+\|w(t_0)\|_{L^2(\R^d)}.
\end{align}
From Proposition \ref{prop:L2-w}, $\|w(t_0)\|_{L^2(\R^d)}$ is only dependent on $N$,  not dependent on $t_0$. Hence $\delta=\delta(N)$. This extends the lifespan to $\R$ and thus proves the global well-posedness.

Lastly, we prove that $u(t)\in \dot H^{s_c}(\R^d)$ for any $t\in \R$. To this end, we first claim that for any $t_0\in \R$, 
\begin{align}\label{est:L2Lr-u}
\|u\|_{L^2_tL^{\frac{dp}{2-p}}_x([t_0,t_0+\delta]\times\R^d)}\lesssim_{N} 1.
\end{align}
Since $u=v+w$, it reduces to show 
$$
\big\|v\big\|_{L^2_tL^{\frac{dp}{2-p}}_x([t_0,t_0+\delta]\times\R^d)}\lesssim_{N} 1;\quad 
\big\|w\big\|_{L^2_tL^{\frac{dp}{2-p}}_x([t_0,t_0+\delta]\times\R^d)}\lesssim_{N} 1.
$$
The first one is followed directly by Lemma \ref{lem:est-v-long}. The second one is followed 
from \eqref{interp-sc}, \eqref{est:w-L2-t0} and the $L^2$ uniform boundedness from Proposition \ref{prop:L2-w}. This gives \eqref{est:L2Lr-u}.

Fixing $t_0\in\R$, suppose that $w(t_0)\in \dot H^{s_c}(\R^d)$, then arguing similarly as the proof of Lemma \ref{lem:local-result-u}, we obtain that for some constants $C_1, C_2>0$, 
\begin{align*}
\big\| u(t)\big\|_{L^\infty_t\dot H^{s_c}_x([t_0,t_0+\delta]\times\R^d)}
\le & \|u(t_0)\|_{\dot H^{s_c}(\R^d)}+C_1\big\| u\big\|_{L^2_tL^{\frac{dp}{2-p}}_x([t_0,t_0+\delta]\times\R^d)}^{p+1}\\
\le & \|u(t_0)\|_{\dot H^{s_c}(\R^d)}+C_2.
\end{align*}
Since $u_0\in \dot H^{s_c}(\R^d)$ and $\delta\sim_N 1$,   by iteration, we have that for any $t>0$, 
$$
\big\| u(t)\big\|_{\dot H^{s_c}(\R^d)}
\lesssim_{N} 1+t.
$$
The negative direction can be treated similarly.  This finishes the proof of Theorem \ref{thm:main01}.

\section*{Acknowledgements}

Part of this work was done while M. Beceanu, A. Soffer and Y. Wu were visiting CCNU (C.C.Normal University) Wuhan, China. The authors thank the institutions for their hospitality and the support. M.B. is partially supported by NSF grant DMS 1700293.  Q.D. is supported by NSFC 11971191 and partially by NSFC 11771165.  A.S is partially supported by NSF grant DMS 01600749 and NSFC 11671163. Y.W. is partially supported by NSFC 11771325 and 11571118. A.S. and Y.W. would like to thank  Ch. Miao for useful discussions.

\end{document}